\documentclass{amsart}

\usepackage{amsmath}
\usepackage{amsfonts}
\usepackage{amssymb}
\usepackage{amsthm}
\usepackage[T1]{fontenc}

\begin{document}

\title[Lower bounds on the number of rational points of Jacobians]{Lower bounds on the number of 
rational points of Jacobians over finite fields and application to algebraic function fields in towers}
\author{S. Ballet}
\address{Aix-Marseille Universit{\'e}, Institut de Math\'{e}matiques
de Luminy\\ case 930, F13288 Marseille cedex 9\\ France}
\email{stephane.ballet@univ-amu.fr}

\author{R. Rolland}
\address{Aix-Marseille Universit{\'e}, Institut de Math\'{e}matiques
de Luminy\\ case 930, F13288 Marseille cedex 9\\ France}
\email{robert.rolland@acrypta.fr}

\author{S. Tutdere}
\address{Orta Do\u gu Teknik  \"Universitesi,  Uygulamal\i\; Matematik Enstit\"us\"u \\ 
 Kriptoloji Laboratuvar\i\;, \"Universiteler Mah. Dumlup\i nar Bul. No:1, 06800 \c Cankaya/Ankara\\ Turkey }    
\email{stutdere@gmail.com}

\date{\today}

\newtheorem{theoreme}{Theorem}[section]
\newtheorem{lemme}[theoreme]{Lemma}
\newtheorem{proposition}[theoreme]{Proposition}
\newtheorem{corollaire}[theoreme]{Corollary}
\theoremstyle{definition}
\newtheorem{conclusion}[theoreme]{Conclusion}
\newtheorem{definition}[theoreme]{Definition}
\newtheorem{remarque}[theoreme]{Remark}
\newtheorem{exemple}[theoreme]{Example}
\renewcommand{\theequation}{\arabic{equation}}
\setcounter{equation}{0}

\begin{abstract}
We give effective bounds for the class number of any algebraic function 
field of genus $g$ defined over a finite field. These bounds depend on the possibly partial information 
on the number of places on each degree $\leq g$. Such bounds are especially useful for estimating
the class number of function fields in towers of function fields over finite
fields. We give examples in the case of asymptotically good towers.
In particular we estimate the class number of function fields which are steps of 
towers having one or several positive Tsfasman-Vladut invariants. Note that the study is not done 
asymptotically, but for each individual step of the towers for which we determine precise 
parameters .
\end{abstract}

\maketitle

\section{Introduction}
\label{sec:intro}

\subsection{General context}\label{subsec:gc}

We recall that the class number $h(F/{\mathbb{F}}_q)$ of an algebraic function field $F/{\mathbb{F}}_q$  
defined over a finite field ${\mathbb{F}}_q$ is the cardinality of the Picard group of $F/{\mathbb{F}}_q$.
 This numerical invariant corresponds to the number of ${\mathbb{F}}_q$-rational 
points of the Jacobian of any curve $X({\mathbb{F}}_q)$  
having $F/{\mathbb{F}}_q$ as algebraic function field. 
Estimating the class number of an algebraic function field is a classical problem. 
By the standard estimates deduced from the results of Weil \cite{weil1} \cite{weil2}, we know that 
$$(\sqrt{q}-1)^{2g}\leq h(F/{\mathbb{F}}_q) \leq (\sqrt{q}+1)^{2g},$$
where $g$ is the genus of $F/{\mathbb{F}}_q$. Moreover, these estimates  hold 
for any Abelian variety. The value of $h(F/{\mathbb{F}}_q)$ is completely determined by the zeta function
of the function field. But generally, we have not enough information to compute the zeta function effectively.
Finding good estimates for the class number  $h(F/{\mathbb{F}}_q)$ from a partial
information on the number of places of degrees $\leq g-1$ is a difficult problem.
For $g=1$, namely for elliptic curves, the class number  is the number of 
${\mathbb{F}}_q$-rational points 
of the curve and this case has been extensively studied. 
So, from now on we assume that $g\geq 2$. 
In \cite{lamd}, Lachaud and Martin-Deschamps proved that $h(F/{\mathbb{F}}_q \geq h_{LMD}$
where
\begin{equation}\label{HLMD}
h_{LMD}= q^{g-1} \frac{(q-1)^2}{(q+1)(g+1)}.
\end{equation}
This estimate do no use any information on the number of places.
Recently in \cite{baro6}, the two first authors gave the following results: 
\begin{theoreme}\label{boundJTNB} 
Let $F/{\mathbb{F}}_q$ be an algebraic function field defined over ${\mathbb{F}}_q$ of genus $g\geq 2$ and 
$h$ the class number of $F/{\mathbb{F}}_q$. Suppose that the numbers $B_1$ and $B_r$ of places of 
degree respectively $1$ and $r$
are such that $B_1\geq 1$ and $B_r\geq 1$. Let us denote by 
$K_1(i,B)$ and $K_2(q,j,B)$ 
the following numbers:
$$K_1(i,B)= 
\left (
\begin{array}{c}
 B+i\\
 B
\end{array}
\right )
\quad \hbox{and}\quad
K_2(q,j,B)= \sum_{i=0}^{j}\frac{1}{q^{i}}
\left (
\begin{array}{c}
B+i-1\\
 B-1
\end{array}
\right )
.$$
Then, the following inequality holds:
$$h \geq h_{BR}$$
where
\begin{equation}\label{minh}
\begin{split}
h_{BR}= 
\frac{(q-1)^2 }{(g+1)(q+1)-B_1} \left [ K_2(q,r-1,B_1) \, q^{g-1} 
\sum_{m=0}^{\left \lfloor \frac{g-2}{r} \right \rfloor-1} \frac{1}{(q^{r})^m}\left (
\begin{array}{c}
B_r+m-1\\
 B_r-1
\end{array}
\right ) \right .\\
+q \, K_2\left(\strut q,(g-2)\mod r,B_1\right) \left (
\begin{array}{c}
 B_r+\left \lfloor \frac{g-2}{r}  \right\rfloor-1\\
 B_r-1
\end{array}
\right )\\
+K_1(r-1,B_1)
\left (
\begin{array}{c}
 B_r+\left \lfloor \frac{g-1}{r} \right \rfloor-1\\
 B_r
\end{array}
\right )\\ 
\left .
+K_1\left(\strut (g-2)\mod r ,B_1\right )
\left (
\begin{array}{c}
 B_r+\left \lfloor \frac{g-1}{r} \right \rfloor-1\\
 B_r-1
\end{array}
\right ) \right ]
\end{split}
\end{equation}
\end{theoreme} 

\begin{theoreme}
Let $F/{\mathbb{F}}_q$ be a function field  of genus $g\geq 2$ defined over a finite field  ${\mathbb{F}}_q$  and 
$h$ be its the class number. Suppose that the numbers $B_1$ and $B_r$ of places of 
degree respectively $1$ and $r$
are such that $B_1>0$ and $B_r>0$ for any $r\in \mathbb{N}$. Let $m_{r}(n)$ be  the quotient of the Euclidean 
division of the integer $n$ 
by the integer $r$.
Then
$$h \geq  \frac{(q-1)^2}{(g+1)(q+1)}q^{g-1}h_r$$
where $h_r$ is defined as follows:

\begin{enumerate}
\item For any $B_r$, let us set 
$$f_r=
\left\{
\begin{array}{lcl}
0 & \hbox{ if } & \frac{g-2}{2}<r\leq g-2\\
1 & \hbox{ if } & r\leq\frac{g-2}{2} \hbox{ and } B_r<q^r\\
\min\left(\lfloor \frac{B_r-q^r}{q^r-1} \rfloor+1,m_r(g-2)-1\right) & \hbox{ if } 
& r\leq \frac{g-2}{2} \hbox{ and } B_r\geq q^r.
\end{array}
\right.
$$
then  $$h_r= \left(\frac{q^{rm_r(g-2)}-1}{q^{r(m_r(g-2)-1)}(q^r-1)}+
\frac{(B_r-1)}{q^r}f_r\right).$$
 \item If $B_r\leq m_r(g-2)$ then
$$h_r=\left(\frac{q^r}{q^r-1} \right)^{B_r-1}.$$
 \item If $B_r > m_r(g-2)$ and $r\leq \frac{g-2}{2}$ then
$$h_r=\left(1+\frac{B_r}{q^r(m_r(g-2)-1)} \right)^{m_r(g-2)-1}.$$
 \item If $B_r+1 \leq (m_r(g-2)-1)(q^r-1)$ then
$$h_r =\left[\left(\frac{q^r}{q^r-1} \right)^{B_r} - B_r \left (\begin{array}{c}
B_r+m_r(g-2)-1\\
 B_r
\end{array}
\right )\left(\frac{1}{q^r}\right)^{m_r(g-2)}\right].$$
\end{enumerate}
\end{theoreme}

Note that the different cases for the value $h_r$ are not mutually exclusive, no case is better 
than the others in absolute terms.
In fact, the optimal choice of $h_r$  depends on the parameters (curve, genus, places of 
degree $r$ etc.). In particular, 
the assertion (a) is always valid but the estimation can be improved by the other 
assertions depending on the values 
of parameters.

However, in all the cases, the above result significantly  improves the lower bounds of 
Lachaud - Martin-Deschamps in \cite{lamd}. 
Moreover, they obtain the following asymptotic result \cite{baro4} 
which reaches some asymptotics for the Jacobian  obtained by 
Tsfasman \cite[Corollary 2]{tsfa} and Tsfasman-Vladut \cite{tsvl} 
(cf.also \cite{tsvlno}):

\begin{theoreme} \label{NewboundhAsymp}
Let ${\mathcal F}/{\mathbb{F}}_q=(F_k/{\mathbb{F}}_q)_{k\geq 0}$ be a sequence of function 
fields over a finite field
${\mathbb{F}}_q$ and ${\mathcal G}/{\mathbb{F}}_{q^r}=(G_k/{\mathbb{F}}_{q^r})_{k\geq 0}$ 
with $G_k=F_k{\mathbb F}_{q^r}$ the constant field extension of $F_k$ for some $r\geq1$. 
Let $g_k$  be the genus, $h(F_k/{\mathbb{F}}_q)$  be the class number, and 
$B_i(F_k/{\mathbb{F}}_q)$ be the number of places of degree $i$  of $F_k/{\mathbb{F}}_q$ for any $i,k\geq0$.   

Let further $\alpha$ be a positive real number. Suppose  that $B_1(F_k/{\mathbb{F}}_q)\geq 1$ 
for some $k\geq 0$ and there exists an $r \geq 1$ such that  
\begin{enumerate}
\item 
$\liminf\limits_{k\rightarrow\infty} \frac{B_r(F_k/{\mathbb{F}}_q)}{g_k} >\alpha$ or
\item
$\frac{1}{r}\cdot \liminf\limits_{k\rightarrow\infty} \sum_{i\mid r}\frac{iB_i(F_k/{\mathbb{F}}_{q})}{g_k} >\alpha.$
\end{enumerate}
Then
$$h(F_k/{\mathbb{F}}_q) > C \left({\left(\frac{q^r}{q^{r}-1}\right)^{\alpha}q}\right)^{g_k}$$
where $C>0$ is a constant with respect to $k$. 
\end{theoreme}

In particular, it happens when they specialize their study
to some families of curves having asymptotically a large number of places of degree $r$ for some value of $r$, 
namely when $\liminf\limits_{g\rightarrow \infty} {B_r(g)}/{g}>0$. 
Recently, in \cite{hesttu} the third authors {\it et al.} explicitly constructed towers 
${\mathcal F}/{\mathbb F}_q=(F_k/{\mathbb F}_q)_{k\geq 0}$ of curves having asymptotically 
a large number of places of degree $r_i$ for several distinct values of $r_i$, namely with 
several positive Tsfasman-Vladut
invariants:  
$$\beta_{r_i}({\mathcal F})=\lim\limits_{k\rightarrow\infty} {B_{r_i}(F_k/{\mathbb{F}}_q)}/{g_k}>0.$$ 

Consequently, we are interested by estimating the class number of these new towers of algebraic function fields 
defined over ${\mathbb F}_q$, which is the main motivation of this paper. Note that the existence of such 
towers  had already been mentioned 
by Lebacque {\it et al.} \cite{leba} but it follows from the class field theory which does not give 
explicit constructions. Moreover, Lebacque \cite[Theorem 7]{leba} obtains an explicit version 
of the Generalized Brauer-Siegel Theorem  which is valid in the case 
of smooth absolutely irreducible Abelian varieties defined over a finite field and for 
the number fields under the Generalized Riemann Hypothesis(GRH). 
Specialized to the case of smooth absolutely irreducible curves  over finite fields, 
this theorem leads to the following 
result:

\begin{theoreme}\label{theohleba}\label{hleba}
For any smooth absolutely irreducible curve $X$ of genus $g$ defined over the finite 
field $\mathbb{F}_r$, one has, as $N \rightarrow \infty$:

$$\sum_{m=1}^{N} {\mathbb P}hi_{r^m}\log\left (\frac{r^{m}}{r^{m}-1}\right ) =\log N+ \gamma+ \log(\aleph_X\log r)+ 
\mathcal{O}(\frac{1}{N})+ g\mathcal{O}(\frac{r^{-N/2}}{N})$$
where ${\mathbb P}hi_{r^m}=\# \{ \mathfrak{p}\in \mid X \mid \mid \deg( \mathfrak{p})=m\} $, $\mid X \mid$ 
denotes the set of closed 
points of $X$ and $\aleph_X$ denotes the residue at $s=1$ of the zeta function $\zeta_X$ of $X$. 
Moreover, the $\mathcal{O}$ constants are effective and do not depend on $X$.

\end{theoreme}

Passing to the limit in the previous result gives the asymptotics of Tsfasman-Vladut \cite{tsvl} \cite{tsvl3}. 
Note also that as 
the constants are effective, this result could lead to effective non-asymptotic lower bounds of the class number $h$. 
However, it is clearly not obvious to obtain these effective bounds. 
On the other hand, very recently Aubry, Haloui, and Lachaud \cite[Proposition 2.1 and 
Theorem 2.4]{auhala1} (cf. also \cite{auhala}), 
in the general context
of Abelian Varieties, deduced new lower bounds on the 
number of points of Jacobians.
in particular, they proved the following:

\begin{theoreme}\label{AHL}
Let $C$ be a curve of genus $g\geq 2$ over ${\mathbb F}_q$ with $N$ rational points. Let $J_C({\mathbb F}_q)$
be the Jacobian of $C$ over ${\mathbb F}_q$. Then the following inequalities hold:
\begin{enumerate}
\item[(I)] \label{I} $$\mid J_C({\mathbb F}_q) \mid \geq M(q)^g \left(  q+1+ \frac{N-(q+1)}{g} \right)^g$$
with $M(q)=\frac{1}{S(h(q))}$ where $S(h)=\frac{h^{1/h-1}}{e\log h^{1/h-1}}$ (and $S(1)=1$) 
denotes the Specht's ratio and $h(q)= \left(  \begin{array}{c}
 q^{1/2}+1\\
q^{1/2}-1
\end{array}
\right )^2$.
\item [(II)]\label{II}
$$\mid J_C({\mathbb F}_q) \mid \geq \frac{q-1}{q^g-1}
\left [ \left(  \begin{array}{c}
 N+2g-2\\
2g-1
\end{array}
\right )  + \sum_{i=2}^{2g-1}B_i\left(  \begin{array}{c}
 N+2g-2-i\\
2g-1-i
\end{array}
\right )\right ]$$ 
\item[(III)]\label{III}
If $N\geq g(q^{\frac{1}{2}}-1)+1$ then
$$\mid J_C({\mathbb F}_q) \mid \geq \left(  \begin{array}{c}
 N+g-1\\
g
\end{array}
\right ) - q\left(  \begin{array}{c}
 N+g-3\\
g-2
\end{array}
\right )$$
\item[(IV)] \label{IV}
$$
\mid J_C({\mathbb F}_q) \mid \geq  \frac{(q-1)^2}{(g+1)(q+1)-N}   
\left [ \left(  \begin{array}{c}
 N+g-2\\
g-2
\end{array}
\right ) + \sum_{i=0}^{g-1}q^{g-1-i}\left(  \begin{array}{c}
 N+i-1\\
i
\end{array}
\right ) \right ]$$
\end{enumerate}
\end{theoreme}

Note that Bound (IV) is exactly the particular case of 
\cite[Theorem 3.1]{baro5} (cf. also Theorem \ref{boundJTNB}) 
obtained with $r=1$.

Let us denote by $h_{AHL}$ the bound (III) which is the most accurate of these bounds with the exception
of bound (IV) already known as previously noted (bound $h_{BR}$): 
\begin{equation}\label{HAHL}
 h_{AHL}=\left(  \begin{array}{c}
 N+g-1\\
g
\end{array}
\right ) - q\left(  \begin{array}{c}
 N+g-3\\
g-2
\end{array}
\right ).
\end{equation}
Note that, Bound (II) gives no significant result
mainly because of the factor $q^g$ in denominator. 

\subsection{New results}\label{subsec:nr}

We remark that Theorem \ref{boundJTNB} in the non-asymptotic case and Theorem \ref{NewboundhAsymp} 
in the asymptotic
case are particularly suitable for  the curves having many points of a given degree $r$. Here, 
we are interested in studying curves 
having  for several fixed distinct degrees a lot of points.
In this paper we give bounds on the class number of an algebraic function field
of one variable defined over the finite field ${\mathbb F}_q$ in the non-asymptotic case, 
namely when the function field is fixed. More precisely, these bounds which are effective bounds 
depend on possibly partial information on the genus $g$ and the number of places 
of each degree $\leq g$.  
In this context, we generalize Theorem \ref{boundJTNB} to obtain more precise bounds 
taking into account  the number of places of several distinct degrees $r_1, r_2,...,r_u$. 
In particular, we obtain:

\begin{theoreme}\label{mainparticular}
Let $F/{\mathbb{F}}_q$ be an algebraic function field defined over 
${\mathbb{F}}_q$ of genus $g\geq 2$ and $h$ its class number. 
For any $r$ such that $1 \leq r \leq g-1$,
let $B_r$ be the number of places of degree $r$ of $F/{\mathbb{F}}_q$.

For any set $D_1$, any set $D_2$,
any finite set of integers $l=\{l_r\}_{r\in D_1 }$ and any finite set of integers
$m=\{m_r\}_{r\in D_2}$
such that
\begin{enumerate}
 \item $D_1 \subseteq \{1,\cdots,g-1\}$;
 \item for any $r \in D_1$ we have $B_r \geq 1$;
 \item $D_2 \subseteq \{1,\cdots,g-2\}$;
 \item for any $r \in D_2$ we have $B_r \geq 1$;
 \item $l_r \geq 0, \hbox{  } \sum_{r\in D_1} rl_r \leq g-1$;
 \item $m_r \geq 0, \hbox{  } \sum_{r\in D_2} rm_r \leq g-2$.
\end{enumerate}
Then $h \geq h_{BRT}$ where
\begin{equation*}
\begin{split}
h_{BRT}= 
\frac{(q-1)^2}{(g+1)(q+1)-B_1} 
\left(
\prod_{r\in D_1} \left(\begin{array}{c}
 B_r+l_r\\
l_r
\end{array}\right) \right .
 \\
\left . +q^{g-1}
\prod_{r\in D_2}\left [ \left( \frac{q^r}{q^r-1}\right)^{B_r}
-B_r\left(
\begin{array}{c}
 B_r+m_r\\
 B_r
\end{array}
\right) \int_{0}^{\frac{1}{q^r}} \frac{(\frac{1}{q^r}-t)^{m_r}}{(1-t)^{B_r+m_r+1}}\rm dt. \right ]
\right).
\end{split}
\end{equation*}
\end{theoreme}

\smallskip

We then apply these bounds to each algebraic function field $F_k/{\mathbb F}_q$ of genus $g_k$ of towers 
having one or several strictly positive Tsfasman-Vladut invariants, namely
for which certain ratios $\frac{B_{r_i}(k)}{g_k}$ have a strictly positive limit when the integer $k$ 
tends to the infinity, where $B_{r_i}(k)$ denotes the number of places of degree $r_i$ of $F_k$. 
Note that the study is not done asymptotically, but for each individual step of the towers for which we determine precise 
parameters: in particular estimations of genus and of number of places of certain degrees for each step.
Moreover, we design new examples of {\it compositum} towers (built from known asymptotically good towers). 
Then, we make a precise study of each step of these towers, included the case of known asymptotically good towers.
All the given examples are chosen in order to cover a large spectrum of different significant cases.    
In particular, we construct a basis of examples of towers of algebraic function fields  illustrating a variety of 
significant situations. Moreover, for each step of these towers, we give information about parameters of concerned 
function field and we show how to best estimate the class numbers depending on the nature of the estimation 
of parameters (exact value, upper and lower bounds, the presence or absence of rational places).

\subsection{Organization of the paper}
In Section 2, we fix the main objective in connection with the fundamental equality (\ref{mainformula}).
 In Section \ref{sec:lower}, we estimate 
the two terms $\Sigma_1=\sum_{n=0}^{g-1}A_n$ and $\Sigma_2=\sum_{n=0}^{g-2}q^{g-1-n}A_n$ of 
the number $S(F/{\mathbb{F}}_q)$ introduced in Subsection \ref{preli}.
This yields in Section \ref{sec:class} new lower and upper bounds on the class number. 
In Section 5, we study lower bounds of the class number  of each step of several towers of algebraic function fields, 
in particular we give some numerical estimations.

\section{Preliminaries}\label{preli}

Throughout this paper we use basic facts and notations as in \cite{stic2}. 
We consider a function field $F/{\mathbb F}_q$ over the finite field with $q$ elements,
of genus $g=g(F)\geq 2$. 
Let $A_n=A_n(F/{\mathbb{F}}_q)$ be the number of effective divisors of degree $n$ of 
 $F/{\mathbb{F}}_q$ and  $h=h(F/{\mathbb{F}}_q)$ the class number of $F/{\mathbb{F}}_q$. 
Let $B_n=B_n(F/{\mathbb{F}}_q)$ the number of places of degree $n$ of $F/{\mathbb{F}}_q$.

Let us set

 $$S(F/{\mathbb{F}}_q)=\sum_{n=0}^{g-1}A_n+\sum_{n=0}^{g-2}q^{g-1-n}A_n \quad 
 \hbox{ and } \quad R(F/{\mathbb{F}}_q)=\sum_{i=1}^{g}\frac{1}{\mid1-\pi_i\mid^2},$$ 
where $(\pi_i,\overline{\pi_i})_{1\leq i \leq g}$ are the reciprocal roots of the numerator 
of the zeta-function $Z(F/{\mathbb{F}}_q,T)$ of $F/{\mathbb{F}}_q$.
By a result due to G. Lachaud and M. Martin-Deschamps \cite{lamd}, we know that
\begin{equation}\label{mainformula}
S(F/{\mathbb{F}}_q)=hR(F/{\mathbb{F}}_q).
\end{equation}

Therefore, in order to find good lower bounds on  $h$, one just needs to find a good lower bound on
$S(F/{\mathbb{F}}_q)$ and  a good upper bound on $R(F/{\mathbb{F}}_q)$. 

It is known by \cite{lamd} that the quantity $R(F/{\mathbb{F}}_q)$  is bounded by the following upper bound: 
 
\begin{equation}\label{R2}
R(F/{\mathbb{F}}_q)\leq \frac{1}{(q-1)^2}  \left(\strut (g+1)(q+1)-B_1(F/{\mathbb{F}}_q)\right).
\end{equation}
The inequality (\ref{R2}) is obtained as follows:
\begin{equation*}
R(F/{\mathbb{F}}_q)=\sum_{i=1}^{g}\frac{1}{(1-\pi_i)(1-\overline{\pi_i})}=
\sum_{i=1}^{g}\frac{1}{1+q-(\pi_i+\overline{\pi_i})}.
\end{equation*}
Multiplying the denominators by the corresponding conjugated quantities, 
we get: 
$$R(F/{\mathbb{F}}_q)\leq \frac{1}{(q-1)^2}\sum_{i=1}^{g}(1+q+\pi_i+\overline{\pi_i}).$$
This last inequality associated to the following formula deduced from the Weil's formulas:
$$\sum_{i=1}^{g}(\pi_i+\overline{\pi_i})=1+q-B_1(F/{\mathbb{F}}_q),$$
gives the inequality (\ref{R2}). 
The inequality  (\ref{R2}) cannot be improved in the general case. 
Remark that in the same way we can prove that
\begin{equation}\label{R3}
R(F/{\mathbb{F}}_q) \geq \frac{1}{(q+1)^2}  \left(\strut (g+1)(q+1)-B_1(F/{\mathbb{F}}_q)\right).
\end{equation}

Hence, in this paper, we propose to study some lower bounds on   
$S(F/{\mathbb{F}}_q)$. 
In this aim, we determine some bounds of the quantities $\sum_1= \sum_{n=0}^{g-1}A_n$ and $\sum_2= 
\sum_{n=0}^{g-2}q^{g-1-n}A_n$  obtained from the number of effective divisors of degree $n\leq g-1$ 
containing in their support 
places of some  fixed distinct degrees  $r_1,r_2,...,r_k \geq 1$.
We deduce bounds on the class number.
 
\section{General  bounds on the sums $\Sigma_1$ and $\Sigma_2$}\label{sec:lower}
In this section, from an exact formula, giving the number of effective divisors of degree n, 
we derive an exact formula for $S(F/{\mathbb{F}}_q)$. 
Unfortunately this exact formula involves some data, which are not necessarily available.
So, we propose methods adapted to cases where some data are not known.
These methods do not give longer exact values of $h$ but only more or less accurate estimates.

\subsection{The basic equations}

In this section, we define the main quantities required for the study of the class number.
We set 
\begin{equation}\label{sigma1}
\Sigma_1=\sum_{n=0}^{g-1}A_n \quad \hbox{and} \quad \Sigma_2=q^{g-1}\sum_{n=0}^{g-2}\frac{A_n}{q^n}.
\end{equation}
Then we can write
\[S=\Sigma_1+\Sigma_2.\]
Let us introduce the following notations that will be used in the whole paper.
\begin{equation}\label{Delta1}
\Delta_1=\{r ~|~ 1 \leq r \leq g-1 \hbox{ and } B_r \geq 1\},
\end{equation}
\begin{equation}\label{Delta2}
\Delta_2=\{r ~|~ 1 \leq r \leq g-2 \hbox{ and } B_r \geq 1\},
\end{equation}
\begin{equation}\label{Deltaprim1}
\Delta'_1=\{r ~|~ 2 \leq r \leq g-1 \hbox{ and } B_r \geq 1\}=\Delta_1 \setminus \{1\},
\end{equation}
\begin{equation}\label{Deltaprim2}
\Delta'_2=\{r ~|~ 2 \leq r \leq g-2 \hbox{ and } B_r \geq 1\}=\Delta_2 \setminus \{1\},
\end{equation}

Let us fix an integer  $n\geq 0$ and set
\begin{equation}\label{Un}
U_n=\left \{b=(b_r)_{r\in \Delta_1} ~|~ b_r \geq 0 \hbox{ and } \sum_{r\in \Delta_1} rb_r = n\right \}.
\end{equation}

First, note that if $B_r \geq 1$ and $b_r \geq 0$ in (\ref{Un}),  the number of solutions of the equation 
$n_1+n_2+ \cdots+ n_{B_r} = b_r$
with integers $\geq 0$ is
\begin{equation}\label{mcoef}
\left (
\begin{array}{c}
 B_r+b_r-1\\
 b_r
\end{array}
\right ).
\end{equation}
Then the number of effective divisors of degree $n$ is given by the following 
result \cite{tsvl} (cf. also \cite{baro5}):

\begin{proposition}\label{propoAn}
The number of effective divisors of degree $n$ of an algebraic function field $F/{\mathbb{F}}_q$ is
\begin{equation} A_n=\sum_{b\in U_n}\left [ \prod_{r\in \Delta_1} \left (
\begin{array}{c}
 B_r+b_r-1\\
b_r
\end{array}
\right )\right ]. \end{equation}
\end{proposition}

Let $D$ be a subset of $\Delta_1$. 
Let $I(D)=(I_r)_{r\in D}$ be a finite sequence of finite subsets $I_r$ of ${\mathbb N}$. We define
the hypercube $C_{I(D)}$ as the following product:
\begin{equation}\label{CI}
C_{I(D)}=\prod_{r\in D} I_r.
\end{equation}
The components of an element $b \in C_I(D)$ will be denoted by $b_r$, namely
$b=(b_r)_{r\in D}$.

\begin{lemme}\label{trivial}
Let $D$ be a subset of $\Delta_1$ and $f$ a map from $D \times {\mathbb N}$ to ${\mathbb N}$ .
Then
\begin{equation*}
\sum_{b\in C_I(D)} \left [\prod_{r\in D}
f(r,b_r)\right]=
\prod_{r\in D}\left[ \sum_{b_r\in I_r}
f(r,b_r).\right]
\end{equation*}
\end{lemme}
\begin{proof}
Denote by $s$ the cardinality of $D$ and denote by $f_r$ the function defined by 
$f_r(b_r)=f(r,b_r)$.
Remark that the previous equation can be written:
$$\sum_{b_{r_1}\in I_{r_1}}\sum_{b_{r_2} \in I_{r_2}} \cdots \sum_{b_{r_s}\in I_{r_s}}
f_{r_1}(b_{r_1})f_{r_2}(b_{r_2})\cdots f_{r_s}(b_{r_s})=$$
$$\left(\sum_{b_{r_1}\in I_{r_1}}
f_{r_1}(b_{r_1})\right)\left(\sum_{b_{r_2}\in I_{r_2}}
f_{r_2}(b_{r_2})\right)\cdots \left(\sum_{b_{r_s}\in I_{r_s}}
f_{r_s}(b_{r_s})\right).$$
The first member is the development of the product of the second member.
\end{proof}

\subsection{The sum $\Sigma_1$}

 In this section, we first give an exact formula of the sum $\Sigma_1$. 
 We then deduce lower and upper bounds for  this sum.
 We obtain the following from Proposition \ref{propoAn} and (\ref{sigma1}): 

\begin{proposition}\label{sommes1}
\begin{equation} \Sigma_1=\sum_{n=0}^{g-1}A_n=\sum_{b\in V_1} \left [\prod_{r\in \Delta_1}
\left(\begin{array}{c}
 B_r+b_r-1\\
b_r
\end{array}\right)\right ],
\end{equation}
where 
$$V_1=\bigcup_{n=0}^{g-1}U_n=\left\{b=(b_r)_{r\in\Delta_1}~|~ b_r\geq 0 \textrm{ and } 
\sum_{r\in\Delta_1} rb_r \leq g-1\right\},$$
\end{proposition}

We can now  establish a lower bound for the quantity $\Sigma_1$:

\begin{theoreme}[Lower bound]\label{lowerboundsum1}
Let $F/{\mathbb{F}}_q$ be an algebraic function field of genus $g\geq 2$. 

\begin{enumerate}
 \item Let $m=(m_r)_{r\in\Delta_1}$ be a finite sequence of integers 
$m_r \geq 0$ such that $\sum_{r\in \Delta_1} rm_r \leq g-1$.  
Then
\begin{equation}\label{lb1}
\Sigma_1 \geq 
\prod_{r\in \Delta_1} \left(\begin{array}{c}
 B_r+m_r\\
m_r
\end{array}\right).
\end{equation}
\item Suppose that $B_1\geq 1$. Let $r_1,r_2,\cdots,r_u$ be distinct elements in $\Delta_1'$. Then
\begin{eqnarray}\label{bl2}
\Sigma_1 &\geq& \left(\begin{array}{c}
 B_1+g-1\\
 g-1
\end{array}\right)\\ 
\nonumber
&+& \sum_{i=1}^u \left(\left(\begin{array}{c}
 B_{r_i}+\left\lfloor \frac{g-1}{r_i}\right\rfloor \\
\left\lfloor \frac{g-1}{r_i}\right\rfloor
\end{array}\right)-1\right)
\left(\begin{array}{c}
 B_{1}+ \left(\strut (g-1)\mod r_i\right)\\
(g-1)\mod r_i
\end{array}\right).
\end{eqnarray}
\end{enumerate}
\end{theoreme}
\begin{proof}

\noindent\begin{enumerate}
\item Let us consider the hypercube $C_{I(\Delta_1)}$ defined
by $I(\Delta_1)=(I_r)_{r\in \Delta_1}$ where $I_r$ is the interval $I_r=\{0,\cdots,m_r\}$. The condition 
$\sum_{r\in \Delta_1}rm_r \leq g-1$ implies that $C_{I(\Delta_1)} \subseteq V_1$. 
Then by using Proposition \ref{sommes1}, the following holds:
\begin{equation*} 
\Sigma_1=\sum_{b\in V_1} \left [\prod_{r\in \Delta_1}
\left(\begin{array}{c}
 B_r+b_r-1\\
b_r
\end{array}\right)\right ] \geq \sum_{b\in C_{I(\Delta_1)}} \left [\prod_{r\in \Delta_1}
\left(\begin{array}{c}
 B_r+b_r-1\\
b_r
\end{array}\right)\right].
\end{equation*}
Now, by Lemma \ref{trivial} we get
\begin{equation*} 
\Sigma_1 \geq \prod_{r\in \Delta_1}\left [\sum_{b_r=0}^{m_r}  
\left(\begin{array}{c}
 B_r+b_r-1\\
b_r
\end{array}\right)
\right]=
\prod_{r\in \Delta_1}\left(\begin{array}{c}
 B_r+m_r\\
m_r
\end{array}\right).
\end{equation*}
\item  For each $0 \leq i \leq u$, we  introduce the hypercube $C_{I_r^{(i)}(\Delta_1)}$ as follows:
$$ I^{(0)}_r= \left\{
\begin{array}{ll}
\{0,\cdots,g-1\} & \hbox{if } r=1, \\
\{0\} & \hbox{if } r>1. 
\end{array}
\right .$$
and for  $i>0$,
$$ I^{(i)}_r= \left\{
\begin{array}{ll}
\{1,\cdots,\left\lfloor \frac{g-1}{r}\right\rfloor\} & \hbox{if } r=r_i, \\
\{0,\cdots,g-1\mod r_i\} & \hbox{if } r=1,\\
\{0\} & \hbox{else}\\
\end{array}
\right .$$
These  hypercubes are pairwise disjoint and all are included in $V_1$.
Then
\begin{eqnarray*}
\Sigma_1 &\geq& \sum_{i=0}^{u} \prod_{r\in \Delta_1}\sum_{b_r \in I_r^{(i)}}
\left (
\begin{array}{c}
 B_r+b_r-1\\
 b_r
\end{array}
\right )
=\sum_{b_1=0}^{g-1}
\left(
\begin{array}{c}
 B_1+b_1-1\\
 b_1
\end{array}
\right )\\
&+&
\sum_{i=1}^{u} \left(\sum_{b_{r_i}=1}^{\left\lfloor \frac{g-1}{r_i}\right\rfloor}
\left( \begin{array}{c}
 B_{r_i}+b_{r_i}-1\\
 b_{r_i}
\end{array}
\right)\right) \left(\sum_{b_{1}=0}^{(g-1) \mod r_i}
\left(\begin{array}{c}
 B_{1}+b_{1}-1\\
 b_{1}
\end{array}\right)\right)\\
&=&\left(\begin{array}{c}
 B_1+g-1\\
 g-1
\end{array}\right)\\
&+& \sum_{i=1}^u \left(\left(\begin{array}{c}
 B_{r_i}+\left\lfloor \frac{g-1}{r_i}\right\rfloor \\
\left\lfloor \frac{g-1}{r_i}\right\rfloor
\end{array}\right)-1\right)
\left(\begin{array}{c}
 B_{1}+ \left(\strut (g-1)\mod r_i\right)\\
(g-1)\mod r_i
\end{array}\right).
\end{eqnarray*}

\end{enumerate}
\end{proof}

\begin{remarque}\label{remarquesum1}
 By using (\ref{bl2}) with $u=1$, 
we obtain the following: 
\begin{eqnarray*}
\Sigma_1 &\geq&  \left(
\begin{array}{c}
 B_1+g-1\\
 g-1
\end{array}
\right)+
\left(
\begin{array}{c}
 B_r+\left\lfloor \frac{g-1}{r_1}\right\rfloor\\
  \left\lfloor \frac{g-1}{r_1}\right\rfloor
\end{array}\right)
\left(\begin{array}{c}
 B_1+\left(\strut (g-1)\mod r_1\right)\\
  (g-1)\mod r_1
\end{array}
\right)\\
&-& 
\left(
\begin{array}{c}
B_1+\left(\strut (g-1)\mod r_1\right)\\
 (g-1)\mod r_1
\end{array}
\right)
\end{eqnarray*}\label{bl3}
which is better than the following formula obtained in \cite{baro5}:

\begin{equation}
\Sigma_1 \geq 
\left(
\begin{array}{c}
 B_r+\left\lfloor \frac{g-1}{r_1}\right\rfloor\\
  \left\lfloor \frac{g-1}{r_1}\right\rfloor
\end{array}\right)
\left(\begin{array}{c}
 B_1+\left(\strut (g-1)\mod r_1\right)\\
  (g-1)\mod r_1
\end{array}
\right) 
\end{equation}\label{bl20}
except in the case $r_1=1$ where the two bounds are the same.
\end{remarque}

\begin{theoreme}[Upper bound]\label{upperboundsum1}
Let $F/{\mathbb{F}}_q$ be an algebraic function field of genus $g\geq 2$.
Then the following holds:
\begin{equation}
\Sigma_1 \leq \prod_{r\in \Delta_1} \left(\begin{array}{c}
 B_r+\left\lfloor \frac{g-1}{r}\right\rfloor\\
\left\lfloor \frac{g-1}{r}\right\rfloor
\end{array}\right).
\end{equation}
\end{theoreme}

\begin{proof}
Let us consider the hypercube $C_{I(\Delta_1)}$ where $I(\Delta_1)=(I_r)_{r\in \Delta_1}$ is such that
for any $r\in \Delta_1 $, the subset $I_r$ is $\{0,\cdots,\left \lfloor \frac{g-1}{r} \right\rfloor$\}.
Hence, $V_1 \subseteq C_{I(\Delta_1)}$ and the following holds:
\begin{equation*}
\begin{split}
 \Sigma_1 \leq \prod_{r\in\Delta_1}\left[ \sum_{b_r=0}^{\left \lfloor \frac{g-1}{r} \right\rfloor}
\left(\begin{array}{c}
 B_r+b_r-1\\
b_r
\end{array}\right)\right]
=\prod_{r\in\Delta_1}
\left(\begin{array}{c}
 B_r+\left\lfloor \frac{g-1}{r_i}\right\rfloor\\
\left\lfloor \frac{g-1}{r_i}\right\rfloor
\end{array}\right)
\end{split}
\end{equation*}
\end{proof}

\subsection{The sum $\Sigma_2$}
In this section, we first give an exact formula of the sum $\Sigma_2$. 
 Then, we deduce lower and upper bounds of  this sum.
 From Proposition \ref{propoAn} and  (\ref{sigma1}), we obtain: 
 
\begin{proposition}
\begin{equation} \Sigma_2=
q^{g-1}\sum_{b\in V_2} \left [\prod_{r\in \Delta_2}\frac{1}{q^{rb_r}}
\left(\begin{array}{c}
 B_r+b_r-1\\
b_r
\end{array}\right)\right ],
\end{equation}
where 
$$V_2=\bigcup_{n=0}^{g-2}U_n=\left\{b=(b_r)_{r\in\Delta_2}~|~ \sum_{r\in\Delta_2} rb_r \leq g-2\right\}.$$
\end{proposition}

\begin{theoreme}[Lower bound]\label{lowerboundsum2}
Let $F/{\mathbb{F}}_q$ be an algebraic function field of genus $g\geq 2$.
\begin{enumerate}
 \item Suppose that $B_1\geq 1$. Let $m=(m_r)_{r\in\Delta_2}$ be a finite sequence of integers 
$m_r \geq 0$ such that $\sum_{r\in \Delta_2} r m_r \leq g-2$. 
Then the following purely multiplicative formula holds:
$$\Sigma_2 \geq q^{g-1}\prod_{r \in \Delta_2}\left [ \sum_{b_r=0}^{m_r} \frac{1}{q^{rb_r}}\left(\begin{array}{c}
B_r+b_r-1\\
b_r
\end{array}\right)\right ].$$
\item Let $r_1,r_2,\cdots,r_u$ be $u$ distinct elements in $\Delta_2'$. 
Then the following holds:

\begin{equation}\label{toto}
\begin{split}
\Sigma_2\geq q^{g-1}\sum_{b_1=0}^{g-2}\frac{1}{q^{b_{1}}}
\left(
\begin{array}{c}
 B_1+b_1-1\\
 b_1
\end{array}
\right )
+\\
q^{g-1}\sum_{i=1}^u \left(\left [ \sum_{b_{1}=0}^{(g-2)\mod r_i} \frac{1}{q^{b_{1}}}
\left(\begin{array}{c}
B_{1}+b_{1}-1\\
b_{1}
\end{array}\right)\right ]\right .\times\\
\left . \left [ \sum_{b_{r_i}=0}^{\left\lfloor\frac{g-2}{r_i} \right\rfloor} \frac{1}{q^{r_ib_{r_i}}}
\left(\begin{array}{c}
B_{r_i}+b_{r_i}-1\\
b_{r_i}
\end{array}\right)-1\right ]\right).
\end{split}
\end{equation}
\end{enumerate}
\end{theoreme}
\begin{proof}
The proof is similar to that of Theorem \ref{lowerboundsum1} with the relation
$\sum_{r\in \Delta_2}rm_r\leq g-2$ instead of $\sum_{r\in \Delta_1}rm_r\leq g-1$. 
\end{proof}

\begin{theoreme}[Upper bound]\label{upperboundsum2}
Let $F/{\mathbb{F}}_q$ be an algebraic function field of genus $g\geq 2$.
Then 
\begin{equation}
\Sigma_2 \leq q^{g-1}\prod_{r \in \Delta_2}
\left [ \sum_{b_r=0}^{\left\lfloor \frac{g-2}{r}\right\rfloor} \frac{1}{q^{rb_r}}\left(\begin{array}{c}
B_r+b_r-1\\
b_r
\end{array}\right)\right ].
\end{equation}
\end{theoreme}
\begin{proof}
The proof is similar to that of Theorem \ref{upperboundsum1}
\end{proof}

\begin{remarque}\label{remQsum2}
The quantity
\begin{equation}\label{Qrs}
Q_{r,s}=\sum_{k=0}^{s} \frac{1}{q^{rk}}
\left(\begin{array}{c}
B_{r}+k-1\\
k
\end{array}\right)
\end{equation}
plays an important role in estimating $\Sigma_2$. Indeed, by Theorem \ref{lowerboundsum2}
and Theorem \ref{upperboundsum2} we have
$$q^{g-1}\prod_{r \in \Delta_2} Q_{r,m_r} \leq \Sigma_2 \leq 
q^{g-1}\prod_{r \in \Delta_2} Q_{r,\left\lfloor \frac{g-2}{r} \right\rfloor}$$
and
$$q^{g-1}\left(Q_{1,g-2}+\sum_{i=1}^u Q_{r_i,\left\lfloor \frac{g-2}{r_i} \right\rfloor}\right)
\leq \Sigma_2 \leq  q^{g-1}\prod_{r \in \Delta_2} Q_{r,\left\lfloor \frac{g-2}{r} \right\rfloor}.$$
\end{remarque}

 \begin{remarque}
  If $s=0$ then $Q_{r,s}=1$.
 \end{remarque}

The estimation of $Q_{r,s}$ leads to some interesting lower and upper bounds for $\Sigma_2$
obtained by injecting the bounds on $Q_{r,s}$ in the formulas of Theorem \ref{lowerboundsum2}.

\begin{proposition}\label{BoundsSigma2}
Let $m=(m_r)_{r\in \Delta_2}$ be a finite sequence of integers such that $m_r \geq 0$ and
$\sum_{r\in\Delta_2}rm_r \leq g-2$. Let $r_1,\cdots,r_u$ be elements of $\Delta_2'$.
Then the following inequality holds:
$$\Sigma_{2,inf}\leq \Sigma_2 \leq \Sigma_{2,sup}$$ 
where
$$\Sigma_{2,inf}=
\left\{
\begin{array}{lcl}
q^{g-1}\prod_{r \in \Delta_2} Q_{r,m_r} \\
\hbox{or} \\
q^{g-1}\left(Q_{1,g-2}+\sum_{i=1}^u Q_{r_i,\left\lfloor \frac{g-2}{r_i} \right\rfloor}\right) 
\end{array}
\right.$$
and 
$$\Sigma_{2,sup}=q^{g-1}\prod_{r \in \Delta_2} Q_{r,\left\lfloor \frac{g-2}{r} \right\rfloor}.$$
Moreover, 
for any $r\geq 1$ and $s\geq 0$ the following holds:
$$Q_{r,s}=\left( \frac{q^r}{q^r-1}\right)^{B_r}-B_r\left (
\begin{array}{c}
 B_r+s\\
 B_r
\end{array}
\right ) \int_{0}^{\frac{1}{q^r}} \frac{(\frac{1}{q^r}-t)^s}{(1-t)^{B_r+s+1}}{\rm d}t.$$
 \end{proposition}
\begin{proof}
We set
$$S_r(X)=\sum_{k=0}^{\infty}X^k \left (
\begin{array}{c}
 B_r+k-1\\
k
\end{array}
\right ), \quad T_r(X,s)=\sum_{k=0}^{s}X^k \left (
\begin{array}{c}
 B_r+k-1\\
k
\end{array}
\right ), \textrm{ and }  $$ 
$$ R_r(X,s)=\sum_{k=s+1}^{\infty}X^k \left (
\begin{array}{c}
 B_r+k-1\\
k
\end{array}
\right ).$$
Let us remark that
$$Q_{r,s}=T_r\left(\frac{1}{q^r},s\right) \quad \textrm{ and }\quad  S_r(X)=\frac{1}{(1-X)^{B_r}}$$
which converges for $|X|<1$ and moreover
$$S_r(X)=T_r(X,s)+R_r(X,s).$$
By the Taylor Formula, we get
$$R_r(X,s)=B_r\left (
\begin{array}{c}
 B_r+s\\
 B_r
\end{array}
\right ) \int_{0}^{X} \frac{(X-t)^s}{(1-t)^{B_r+s+1}}{\rm d}t.
$$
Then
$$Q_{r,s}=\left( \frac{q^r}{q^r-1}\right)^{B_r}-B_r\left (
\begin{array}{c}
 B_r+s\\
 B_r
\end{array}
\right ) \int_{0}^{\frac{1}{q^r}} \frac{(\frac{1}{q^r}-t)^s}{(1-t)^{B_r+s+1}}{\rm d}t.$$
\end{proof}

\begin{remarque}
  Note that a lower bound of $\Sigma_2$ always can be computed in practice directly from 
  Proposition  \ref{BoundsSigma2} because it is sufficient to know a subset of $\Delta_1$ 
  and the corresponding $B_r$ (or a lower bound for each corresponding $B_r$)
  while the upper bound can not be always computed easily because it need to know 
  precisely the set $\Delta_1$ and all the $B_r$ of the curve which is difficult when the genus 
is large for example.
 \end{remarque}

\begin{remarque} \label{remarquesum2}
Note that if we specialize to the case $r=1$ without knowing information about the other degrees, we obtain 
$\Sigma_{2,inf}=q^{g-1}Q_{1,g-2}$ and thus the quantity $\Sigma_{2,inf}$ is equal to 
$\Sigma_{2}$ in \cite[Formula (2.6)]{baro5} by Remark \ref{remQsum2} and Proposition \ref{BoundsSigma2}. 
\end{remarque}

\section{Bounds on the class number}\label{sec:class}
In this section, using the lower and upper bounds obtained in the previous section, 
we derive lower ands upper bounds on the class number of an algebraic function field.
 
Let $F/{\mathbb{F}}_q$ be an algebraic function field defined over 
${\mathbb{F}}_q$ of genus $g\geq 2$. For any $r$ such that $1 \leq r \leq g-1$
let $B_r$ be the number of places of degree $r$ of $F/{\mathbb{F}}_q$ and $h$ 
the class number of $F/{\mathbb{F}}_q$.

It is known (cf. \cite{lamd}) that the class number $h$ is
$$h=\frac{S(F/{\mathbb F}_q)}{R(F/{\mathbb F}_q)}=\frac{\Sigma_1+\Sigma_2}{R(F/{\mathbb F}_q)}.$$
Then, we have

\begin{theoreme}\label{MainTheo}
Let $F/{\mathbb{F}}_q$ be an algebraic function field defined over 
${\mathbb{F}}_q$ of genus $g\geq 2$ and $h$ be its class number. 
For any set $D_1 \subseteq \Delta_1$, any set $D_2 \subseteq \Delta_2$,
any finite set of integers $l=\{l_r\}_{r\in \Delta_1 }$ and any finite set of integers
$m=\{m_r\}_{r\in \Delta_2}$
such that 
$$l_r \geq 0, \hbox{  } \sum_{r\in \Delta_1} rl_r \leq g-1$$
and
$$m_r \geq 0, \hbox{  } \sum_{r\in \Delta_2} rm_r \leq g-2,$$
we have the following
\begin{equation*}
 \frac{(q-1)^2(\Sigma_{1,inf}+\Sigma_{2,inf})}{(g+1)(q+1)-B_1}\leq h 
\leq \frac{(q+1)^2(\Sigma_{1,sup}+\Sigma_{2,sup})}{(g+1)(q+1)-B_1},
\end{equation*}
where
$$\Sigma_{1,inf}=
\left\{
\begin{array}{lcl}
(a) \quad \prod_{r\in D_1} \left(\begin{array}{c}
 B_r+l_r\\
l_r
\end{array}\right) \\
\hbox{or} \\
(b) \quad \left(\begin{array}{c}
 B_1+g-1\\
 g-1
\end{array}\right)+\\ \sum_{r\in D_1} \left(\left(\begin{array}{c}
 B_{r}+\left\lfloor \frac{g-1}{r}\right\rfloor \\
\left\lfloor \frac{g-1}{r}\right\rfloor
\end{array}\right)-1\right)
\left(\begin{array}{c}
 B_{1}+ \left(\strut (g-1)\mod r\right)  \\
(g-1)\mod r
\end{array}  \right )\\ \hbox{ with } B_1\geq 1,
\end{array}
\right .
$$
$$\Sigma_{1,sup}=\prod_{r\in \Delta_1} \left(\begin{array}{c}
B_r+\left\lfloor \frac{g-1}{r}\right\rfloor\\
\left\lfloor \frac{g-1}{r}\right\rfloor
\end{array}\right),
$$
and
$$\Sigma_{2,inf}=
\left\{
\begin{array}{lcl}
(c) \quad q^{g-1}\prod_{r\in D_2} Q_{r,m_r} \\
\hbox{or} \\
(d) \quad q^{g-1}\left(Q_{1,g-2}+\sum_{r\in D_2} Q_{r,\left\lfloor \frac{g-2}{r} \right\rfloor}\right) 
\hbox{ with } B_1\geq 1,
\end{array}
\right .
$$
and
$$\Sigma_{2,sup}=q^{g-1}\prod_{r\in \Delta_2} Q_{r,\left\lfloor \frac{g-2}{r} \right\rfloor}$$
where the quantity $Q_{r,s}$ is defined and estimated in Proposition \ref{BoundsSigma2}. 
\end{theoreme}

Note that the formulas (a) and (c) can be used in all the cases and (b) and (d) only if $B_1\geq 1$. 
Note also that the quantities $\Sigma_{1,inf}$ and $\Sigma_{2,inf}$ can be chosen respectively among 
two bounds whose the quality depends on the studied situation.

The following theorem is a direct consequence of 
Theorem \ref{MainTheo} and Proposition \ref{BoundsSigma2} (a).

\begin{theoreme}\label{mainmain}
Let $F/{\mathbb{F}}_q$ be an algebraic function field defined over 
${\mathbb{F}}_q$ of genus $g\geq 2$ and $h$ its class number. 
For any $r$ such that $1 \leq r \leq g-1$,
let $B_r$ be the number of places of degree $r$ of $F/{\mathbb{F}}_q$.

For any set $D_1$, any set $D_2$,
any finite set of integers $l=\{l_r\}_{r\in D_1 }$ and any finite set of integers
$m=\{m_r\}_{r\in D_2}$
such that
\begin{enumerate}
 \item $D_1 \subseteq \{1,\cdots,g-1\}$;
 \item for any $r \in D_1$, we have $B_r \geq 1$;
 \item $D_2 \subseteq \{1,\cdots,g-2\}$;
 \item for any $r \in D_2$, we have $B_r \geq 1$;
 \item $l_r \geq 0, \hbox{  } \sum_{r\in D_1} rl_r \leq g-1$;
 \item $m_r \geq 0, \hbox{  } \sum_{r\in D_2} rm_r \leq g-2$.
\end{enumerate}
Then $h \geq h_{BRT}$ where
\begin{equation}\label{HBRT}
\begin{split}
h_{BRT}=
\frac{(q-1)^2}{(g+1)(q+1)-B_1} 
\left(
\prod_{r\in D_1} \left(\begin{array}{c}
 B_r+l_r\\
l_r
\end{array}\right) \right .
 \\
\left .+q^{g-1}
\prod_{r\in D_2}\left [ \left( \frac{q^r}{q^r-1}\right)^{B_r} 
-B_r\left(
\begin{array}{c}
 B_r+m_r\\
 B_r
\end{array}
\right) \int_{0}^{\frac{1}{q^r}} \frac{(\frac{1}{q^r}-t)^{m_r}}{(1-t)^{B_r+m_r+1}}\rm dt. \right ]
\right).
\end{split}
\end{equation}
\end{theoreme}

\begin{remarque}\label{comparaisontheorique}
Note that if we specialize to the case $r=1$ without knowing information about other degrees, 
the bound of Theorem \ref{MainTheo} is exactly the same 
as the bound obtained in  \cite[Theorem 3.1]{baro5} because of Remarks \ref{remarquesum1}  and 
\ref{remarquesum2}.      
\end{remarque}
 
\section{ Lower bounds on the class number in towers } 

In this section we study the class number of few recursive towers ${\mathcal F}/\mathbb{F}_q=(F_k/{\mathbb F}_q)$ 
defined over different finite fields $\mathbb{F}_q$. In particular, we consider towers with
$B_{r}(F_k/{\mathbb F}_q)>0$ for various $r\geq1$.  
We study precisely the following parameters of each step $F_k/{\mathbb F}_q$ of those towers:  
the genus, the number of places of certain degrees, and the class number. 
Moreover, we study the possibly existence of at least a place of degree one.
Let us first recall few definitions:

\begin{definition}\label{tower}

\noindent\begin{enumerate} 
\item   
A tower over ${\mathbb F}_q$ is an infinite sequence ${\mathcal F}/{\mathbb F}_q=(F_k/{\mathbb F}_q)_{k\geq 0}$ 
of function fields $F_k/{\mathbb F}_q$ with the following properties:
\begin{itemize}
\item[(a)] for all $k\geq 0$, the field ${\mathbb F}_q$ is algebraically closed in $F_k$ 
and each extension $F_{k+1}/F_k$ is finite separable such that $F_k\subsetneqq F_{k+1}$.
\item[(b)]  The genera $g(F_k)\to \infty$ as $k\to \infty$. 
\end{itemize}
\item Let ${\mathcal F}/{\mathbb F}_q=(F_k/{\mathbb F}_q)_{k\geq 0}$ be a tower and 
$f(X,Y)\in {\mathbb F}_q[X,Y]$ be a non-constant polynomial. Suppose that there exist elements 
$x_k\in F_k$ (for $k\geq 0$) such that
$$ F_{k+1}=F_k(x_{k+1}) \textrm{ with } f(x_k,x_{k+1})=0 \textrm{ for all } n\geq 0.$$
Then we say that the tower ${\mathcal F}/{\mathbb F}_q$ is recursively defined by the polynomial $f(X,Y)$. 
In this case, we call that ${\mathcal F}/{\mathbb F}_q$ is a recursive tower.
\item Let ${\mathcal F}/{\mathbb F}_q=(F_k/{\mathbb F}_q)_{k\geq 0}$ be a tower and $E$ be 
a finite separable extension of $F_0$. Suppose that ${\mathcal E}/{\mathbb F}_q=(E_k/{\mathbb F}_q)_{k\geq 0}$, 
with $E_k:=EF_k$ is a  tower such that $E/F_0$ and $F_k/F_0$ are linearly disjoint for all $k\geq 0$. 
Then we call ${\mathcal E}/{\mathbb F}_q$ the composite tower of ${\mathcal F}/{\mathbb F}_q$ with $E$. 
When ${\mathcal F}/{\mathbb F}_q$ and $E$ are clear, we just say that ${\mathcal E}/{\mathbb F}_q$ is a composite tower. 
\item We will call a tower ${\mathcal F}/{\mathbb F}_q=(F_k/{\mathbb F}_q)_{k\geq 0}$ a quadratic 
tower if $[F_{k+1}:F_k]=2$ for all $k\geq 0$.  
\end{enumerate}
\end{definition}

\begin{remarque}
Let ${\mathcal F}/{\mathbb F}_q=(F_k/{\mathbb F}_q)_{k\geq 0}$ be a recursive tower 
defined by the polynomial $f(X,Y)\in {\mathbb F}_q[X,Y]$ and ${\mathcal E}/{\mathbb F}_q$ 
be a composite tower of ${\mathcal F}/{\mathbb F}_q$ with an extension $E$ of $F_0$. 
Then ${\mathcal E}/{\mathbb F}_q$ is a recursive tower defined by the same polynomial $f(X,Y)$.  
\end{remarque}

Let ${\mathbb P}(F)$ be the set of places of $F/{\mathbb{F}}_q$ and let $v_P$ be the discrete 
valuation associated to the place $P\in {\mathbb P}(F)$. 
Let $E/F$ be a finite separable extension of the function field $F$ over ${\mathbb F}_q$, 
$P\in \mathbb{P}(F)$, $Q\in \mathbb{P}(E)$ 
be a place lying above $P$. Then we write $Q|P$ and denote $e(Q|P)$ the ramification index of $Q|P$, 
$f(Q|P)$ the relative degree of $Q|P$, $d(Q|P)$ the different exponent of $Q|P$ and $Diff(E/F)$ the different of $E/F$. 
We denote by $(x=a)$  the place of the rational function field ${\mathbb F}_q(x)$ 
which is a zero of $x-a$ for $a\in {\mathbb F}_q$, 
or the pole of $x$ for  $a=\infty$.

\begin{definition}
Let ${\mathcal F}=(F_n)_{n\geq 0}$ be a tower over ${\mathbb F}_q$. Then
\[R({\mathcal F}):=\{P\in  {\mathbb P}(F_0) \mid P \textrm{ is ramified in } F_k \textrm{ for some } k\geq 0 \}\]
is called the ramification locus of ${\mathcal F}$. The set
\[Split({\mathcal F}):=\{P\in  {\mathbb P}(F_0) \mid P \textrm{ splits completely in } F_k  \textrm{ for all $k\geq 1$}\}\]
is called the splitting locus  of ${\mathcal F}$. 
\end{definition}

For a tower ${\mathcal F}/{\mathbb F}_q$, let $N=N({\mathcal F}/\mathbb{F}_q)$ be a set of integers $r\geq 1$ 
such that $\beta_{r}({\mathcal F}/{\mathbb F}_q)>0$.

\subsection{Case $N=\{r\}$}

In this section we give some examples of towers having only one positive invariant $\beta_r$ or  
having only one known positive invariant $\beta_r$ for any  $r\geq 1$. These examples are obtained 
from asymptotically good towers with respect to the rational places over ${\mathbb F}_{q^r}$, i.e., with $\beta_1>0$ 
or from the descent of such towers on the ground field ${\mathbb F}_q$.

\subsubsection{Example 1}

This example concerns the first tower  of Garcia-Stichtenoth \cite{gast} reaching 
the Drinfeld-Vladut bound  over ${\mathbb F}_{q^r}$ where $q^r$ is a square.  
This tower is here denoted by $\mathcal{H}/{\mathbb F}_{q^r}$.
We study bounds of the class number of all the steps of this tower as well as those of 
its descent tower $\mathcal{F}/{\mathbb F}_q$ to ${\mathbb F}_{q}$, which reaches the Drinfeld-Vladut bound of order $r$. 
Note that in these two cases, we know the exact value of the genus of each step. 
But in the case of the tower $\mathcal{H}/{\mathbb F}_{q^r}$, it is interesting to see that 
we obtain estimations of class number from the value (or from lower bound) of the number of places of degree one. 
However, in the case of the tower $\mathcal{F}/{\mathbb F}_q$, we can obtain estimations of class numbers from only
the value (or from lower bound) of the quantity $\sum_{i\mid r}iB_i(F_k)$. 
These two towers illustrate two different 
situations where one can estimate class numbers.

\begin{proposition}\label{optimal.exm}
Let $q$ be a prime power and $r\geq 1$ be an integer such that $q^r$ is a square. 
Consider the tower $\mathcal{H}/{\mathbb F}_{q^r}=(H_k/{\mathbb F}_{q^r})_{k> 0}$ 
defined recursively by the polynomial
\begin{equation}\label{GStower} 
f(X,Y)=Y^{q^{r/2}}X^{q^{r/2}-1}+Y-X^{q^{r/2}} \in {\mathbb F}_q[X,Y]
\end{equation}
and the rational function field $F_0=\mathbb{F}_{q}(x_0)$.
Then the tower $\mathcal{H}/{\mathbb F}_{q^r}$ and its descent $\mathcal{F}/{\mathbb F}_q=(F_k/{\mathbb F}_q)_{k\geq 0}$  
which is such that $\mathcal{H}/{\mathbb F}_{q^r}=\mathcal{F}/{\mathbb F}_q\otimes_{{\mathbb F}_q} {\mathbb F}_{q^r}$ 
verify the following
properties:
\begin{itemize}
 \item[(i)] $B_1(F_k)\geq 1$;
 \item[(ii)] a) if $q\equiv 1 \mod 2$ then for all  $k\geq 2$, 
 $$\sum_{i\mid r}iB_i(F_k)=(q^r-1)q^{rk/2}+2q^{r/2}=B_1(H_k).$$\\
 b) if $q\equiv 0 \mod 2$ then $$\sum_{i\mid r}iB_i(F_2)=(q^r-1)(q^r)+2q^{r/2}=B_1(H_2),$$
    $$\sum_{i\mid r}iB_i(F_3)=(q^r-1)q^{3r/2}+q^r+q^{r/2}=B_1(H_3),$$
    $$\sum_{i\mid r}iB_i(F_k)=(q^r-1)q^{rk/2}+2q^r = B_1(H_k) \hbox{ for all }k\geq 4.$$
  \item[(iii)] The genus $g(H_k)=g(F_k)$ and 
$$
g(F_k)=\left\{
\begin{array}{ll}
 (q^{r/2}+1)q^{rk/2}-(q^{r/2}+2)q^{rk/4}+1 & \hbox{if }  k \hbox{ even}\\
 (q^{r/2}+1)q^{rk/2}-\frac{1}{2}(q^r+3q^{r/2}+1)q^{(rk-r)/4}+1& \hbox{if } k \hbox{ odd}.                         
\end{array}\right .
$$
\item[(iv)]  $\beta_r({\mathcal F}/{\mathbb F}_{q})=\frac{q^{r/2}-1}{r}$.
\end{itemize}
\end{proposition}
\begin{proof}
The tower $\mathcal{H}/{\mathbb F}_{q^r}=(H_k/{\mathbb F}_{q^r})_{k\geq 0}$ is the well-known Garcia-Stichtenoth 
tower over ${\mathbb F}_{q^r}$  studied in \cite{gast}.
 Since $\mathcal{H}/{\mathbb F}_{q^r}$ is the constant field extension of ${\mathcal F}/{\mathbb F}_q$, 
the assertions $(ii)$ and $(iii)$ are respectively straightforward consequences 
of Remark 3.4 and Theorem 2.10 in \cite{gast}. 
Assertion $(iv)$ follows from assertions $(ii)$ and $(iii)$. 
To prove (i), we consider the pole $P_{\infty}$ of  $x_0$ in $H_0/\mathbb{F}_{q^r}$. 
By \cite[Lemma 2.1]{gast}, the place $P_{\infty}$ is totally ramified in $H_k/H_0$ for all $k\geq 0$.
Since $P_{\infty}$ is invariant under the action of the Galois group 
$Gal(\mathbb{F}_{q^r}/\mathbb{F}_{q})$, its restriction $P_{\infty}\cap F_0$ in $F_0/\mathbb{F}_{q}$ 
 is then totally ramified in $F_k/\mathbb{F}_q$, and so $(i)$ follows.
 \end{proof}
 Let us recall the following fact:
 \begin{remarque} 
 Let ${\mathcal F}/{\mathbb F}_q=(F_k/{\mathbb F}_q)_{k\geq 0}$ be a sequence with $g(F_k)\to \infty$ as $k\to\infty$. 
 Consider the constant field extension ${\mathcal F}/{\mathbb F}_{q^r}=(F_k{\mathbb F}_{q^r}/{\mathbb F}_{q^r})_{k\geq 0}$  
of ${\mathcal F}/{\mathbb F}_q$, for some $r\geq 1$. 
 It is known from \cite{baro5} that
\[\beta_r({\mathcal F}/{\mathbb F}_q)=\frac{q^{r/2}-1}{r}\; \textrm{ if and only if } \; \beta_1({\mathcal F}/{\mathbb F}_{q^r})=q^{r/2}-1.\]

\end{remarque}

Now, let us give an example concerning the estimation of the class number in all the steps of the towers 
$\mathcal{H}/{\mathbb F}_{q^r}=(H_k/{\mathbb F}_{q^r})_{k\geq 0}$ and 
$\mathcal{F}/{\mathbb F}_q=(F_k/{\mathbb F}_q)_{k\geq 0}$ in the case of $q=2$ 
and $r=2$. 

\begin{proposition}\label{propohgastsur4et2}
The steps of the towers 
$\mathcal{H}/{\mathbb F}_{q^2}=(H_k/{\mathbb F}_{q^2})_{k\geq 0}$ and 
$\mathcal{F}/{\mathbb F}_q=(F_k/{\mathbb F}_q)_{k\geq 0}$ 
defined in Proposition \ref{optimal.exm} with $q=2$ and $r=2$ have  class number such that
$$h(H_k/{\mathbb F}_{q^2})\geq h_{BRT1,k}$$ 
and $$h(F_k/{\mathbb F}_q)\geq h_{BRT2,k}$$
where
\begin{equation*}
\begin{split}
h_{BRT1,k}= 
\frac{(q^r-1)^2}{(g_k+1)(q^r+1)-B_1(H_k)} 
\bigg(
 \left(\begin{array}{c}
 B_1(H_k)+g_k-1\\
B_1(H_k)
\end{array}\right)
 \\
 + 
q^{r(g_k-1)}\left [ \left( \frac{q^r}{q^r-1}\right)^{B_1(H_k)}
-  \right. \\ \left. B_1(H_k)\left(
\begin{array}{c}
 B_1(H_k)+g_k-2\\
 B_1(H_k)
\end{array}
\right) \int_{0}^{\frac{1}{q^r}} \frac{(\frac{1}{q^r}-t)^{g_k-2}}{(1-t)^{B_1(H_k)+g_k-1}}\rm dt. \right ]
\biggr)
\end{split}
\end{equation*}
and
 \begin{equation*}
\begin{split}
h_{BRT2,k}=
\frac{(q-1)^2}{(g_k+1)(q+1)-B'_1(F_k)} 
\bigg(
\prod_{i=1}^2 \left(\begin{array}{c}
 B'_i(F_k)+m_i\\
B'_i(F_k)
\end{array}\right)
 \\
+ 
q^{g_k-1} \prod_{i=1}^2 \left [ \left( \frac{q^i}{q^i-1}\right)^{B'_i(F_k)} 
-\right .\\ \left .
B'_i(F_k)\left(
\begin{array}{c}
 B'_i(F_k)+m_i\\
 B'_i(F_k)
\end{array}
\right) \int_{0}^{\frac{1}{q^i}} \frac{(\frac{1}{q^i}-t)^{m_i}}{(1-t)^{B'_i(F_k)+m_i+1}}\rm dt. \right ]
\biggr)
\end{split}
\end{equation*}
with $B'_1(F_k)=2$, $B'_2(F_k)=\lfloor \frac{B_1(H_k)-2}{2}\rfloor$ and $\sum_{i=1}^2im_i\leq g_k-2$ for all $k\geq2$.
\end{proposition}

\begin{proof}
The bound $h_{BRT1,k}$ follows directly from Theorem \ref{mainmain}. 
By Section \ref{preli}, it is clear that the lower bounds of the class number $h$ are increasing with the number 
of effective divisors constructed with places of degree less or equal to a fixed degree. Here, we do not know $B_1(F_k)$ 
nor $B_2(F_k)$ but we know the value of the quantity $B_1(F_k)+2B_2(F_k)$. This quantity is even then $B_1(F_k)$ is even. 
Hence, if we replace $B_1(F_k)\geq 2$ by $B'_1(F_k)=2$ 
and $B_2(F_k)$ by $B'_2(F_k)= \left \lfloor \frac{(B_1(F_k)+2B_2(F_k))-2}{2} \right \rfloor $, we obtain a lower bound on the number 
of effective divisors and then a lower bound for $h_{BRT2,k}$.
\end{proof}

\medskip

{\bf Numerical estimations:} 

\medskip

Here are some effective numerical estimations of  the class numbers of some steps of the towers 
$\mathcal{H}/{\mathbb F}_{4}$ and $\mathcal{F}/{\mathbb F}_2$ given in Proposition \ref{propohgastsur4et2}.

\medskip

More precisely, the following table gives estimations of the class number of the first steps 
of the original tower  $\mathcal{H}$ of Garcia-Stichtenoth 
reaching the Drinfeld-Vladut bound. We make these estimations with $q=4$ and we only use the places of 
degree one. Moreover, recall that we denote by $h_{AHL}$ the adapted bound in this context extracted from Theorem \ref{AHL}  
(namely bound (III)) obtained by Aubry, Halaoui, and Lachaud. In particular, we see that our bound $h_{BRT}$ 
of Proposition \ref{propohgastsur4et2} is better than $H_{AHL}$ for $k\geq 3$. Note also that in this context  the bound 
$h_{BR}$ obtained in 
\cite[Theorem 3.1]{baro5} gives the same estimations because in the particular case $r=1$, these bounds are the same as those 
given by  $h_{BRT}$ (cf. Remark \ref{comparaisontheorique}).
 
\medskip

\begin{center}
\begin{tabular}{|l|c|c|c|c|c|c|c|clr}
\hline
$q^2$  & step k & $g_k$ &  $B_1(H_k/{\mathbb F}_q)$ & $m_1$ & $h_{BRT1,k}=h_{BR}$ & $h_{AHL}$   \\
\hline
4 &   2 &    5 &     16 &    3  &  7434  & 12240\\
 \hline
4 &     3  &    14  &  30 &     12   & 64 916 794 126   &    46 784 340 680\\
 \hline     
4 &      4       &  33 &    56    &    31 & $1.43 \times 10^{25}$   &  $0.075 \times 10^{25}$ \\
 \hline      
\end{tabular}
\end{center}

\medskip

Now, in the following table, we give estimations of the class number of the first steps 
of the descent tower  of Garcia-Stichtenoth 
reaching the Drinfeld-Vladut bound of order $r=2$. We make these estimations with $q=2$ 
and we only use the places of 
degree one and two. Note that for any $k$, Condition $B_1(H_k/{\mathbb F}_q)\geq g_k(q^{1/2}-1)$ 
of Theorem 2.4 in \cite{auhala1} (or \cite{auhala}) 
is not  satisfied. Hence, we must not use Bound (III) of Theorem \ref{AHL}. 
Moreover, recall that it is clear that Bound (II) gives no significant result
because of the factor $q^g$ in the denominator.

\medskip

 \begin{center}
\begin{tabular}{|l|c|c|c|c|c|c|c|c|}
\hline
q  & step k & $g_k$ &   $B'_1(F_k/{\mathbb F}_q)$ &  $B'_2(F_k/{\mathbb F}_q)$ & $m_1$ & $m_2$ & $h_{BRT2,k}$   \\
\hline
2 &    2&   5 &     2 & 7 &   1 &  1&  7 \\
 \hline
2 &     3  &     14 &  2 &   14  &   2 &   5  &  21257  \\
 \hline     
2 &     4     &  33 &     2   &   27  &  5    & 13 & 343 733 443 618 \\
 \hline      
 \end{tabular}
 \end{center}

\medskip 
 
 \begin{remarque}
Recall that for $H_k/{\mathbb F}_{4}$, $$\Sigma_{1,k,inf}= \left(\begin{array}{c}
 B_1(H_k)+g_k-1\\
B_1(H_k)
\end{array}\right)$$ and 
\begin{equation*}
\begin{split}
\Sigma_{2,,k,inf}=q^{g_k-1}\left [ \left( \frac{q^r}{q^r-1}\right)^{B_1(H_k)} - \right .\\
\left . B_1(H_k)\left(
\begin{array}{c}
 B_1(H_k)+g_k-2\\
 B_1(H_k)
\end{array}
\right) \int_{0}^{\frac{1}{q^r}} \frac{(\frac{1}{q^r}-t)^{g_k-2}}{(1-t)^{B_1(H_k)+g_k-1}}\rm dt \right ].
\end{split}
\end{equation*}

By our numerical simulations in the tower $\mathcal{H}/{\mathbb F}_{q^r}$, 
we have $\frac{\Sigma_{1,k,inf}}{\Sigma_{2,k,inf}}\leq 0.0054$ for all $k\geq 4$ 
and in fact, this ratio tends to zero when $k$ tends to the infinity. Moreover, let us set 
$$C_k=\left( \frac{q^r}{q^r-1}\right)^{B_1(H_k)}$$ and 
$$D_k=B_1(H_k)\left(
\begin{array}{c}
 B_1(H_k)+g_k-2\\
 B_1(H_k)
\end{array}
\right) \int_{0}^{\frac{1}{q^r}} \frac{(\frac{1}{q^r}-t)^{g_k-2}}{(1-t)^{B_1(H_k)+g_k-1}}\rm dt.$$ Then the ratio 
$\frac{D_k}{C_k}\leq 0.01$ for all $k\geq 4$. In fact, this ratio tends to zero when $k$ tends to the infinity. 
The situation 
in the tower $\mathcal{F}/{\mathbb F}_q$ is similar. Moreover, in the numerical estimations 
concerning the tower $\mathcal{F}/{\mathbb F}_q$, we have 
specialized $l_i=m_i$ for $i=1,2$ from the general formula in Theorem \ref{MainTheo}. Indeed, the realized numerical 
simulations on few steps give no improvement of the bounds for the class number of $F_k$ by distinguishing $l_i$ and 
$m_i$ in the quantities 
$$\Sigma_{1,k}=\prod_{i=1}^2 \left(\begin{array}{c}
 B_i(F_k)+l_i\\
B_i(F_k)
\end{array}\right)$$ and 

$$\Sigma_{2,k}=\prod_{i=1}^2 \left [ \left( \frac{q^i}{q^i-1}\right)^{B'_i(F_k)}
-B'_i\left(
\begin{array}{c}
 B'_i(F_k)+m_i\\
 B'_i(F_k)
\end{array}
\right) \int_{0}^{\frac{1}{q^i}} \frac{(\frac{1}{q^i}-t)^{m_i}}{(1-t)^{B'_i(F_k)+m_i+1}}\rm dt \right ].$$   
\end{remarque}

\subsubsection{Example 2}

This example concerns a tower defined over cubic finite fields constructed by  Bassa, 
Garcia and Stichtenoth  \cite{bagast}. 
We study bounds of the class number of all the steps of this tower as well as 
those of its descent tower over ${\mathbb F}_{q}$. 
Note that in these two cases, we only know lower and upper bounds for the genus of each step. 
Moreover, even if we do not know whether there exist rational places in each step of 
the descent tower over ${\mathbb F}_{q}$, we can give good estimations for the class number 
of all the steps of this tower.
 
\begin{proposition}\label{param2} 
Let $q$ be a prime power and
$\mathcal{H}/{\mathbb F}_{q^3}=(H_k/{\mathbb F}_{q^3})_{k\geq 0}$ the tower
defined recursively by the polynomial 
\begin{eqnarray}\label{q3.polyn}
f(X,Y)=(Y^q-Y)^{q-1}+1+\frac{X^{q(q-1)}} {(X^{q-1}-1)^{q-1}} \in {\mathbb F}_q[X,Y]. 
\end{eqnarray}
 The tower $\mathcal{H}/{\mathbb F}_{q^3}$ and the descent tower $\mathcal{F}/{\mathbb F}_q=(F_k/{\mathbb F}_q)_{k\geq 0}$  
defined by its constant field extension 
 $\mathcal{H}/{\mathbb F}_{q^3}=\mathcal{F}/{\mathbb F}_q\otimes_{{\mathbb F}_q} {\mathbb F}_{q^3}$ have the following properties:
 
\begin{itemize}
 \item[(i)] $B_1(H_k)\geq 3B_3(F_k)$   with $B_3(F_k)\geq \frac{(q^3-q^2-q+1)q^{k+1}}{3}$ for all $k\geq 1$.
 \item[(ii)] for all $k\geq 2$, one has that
 \begin{equation}\label{q3}
 \frac{1}{2}\big(q^4+q^3-6q^2+4q)q^{k-1}-q^3+q^2+1\leq g(F_k)\leq \frac{(q^2+q-2)q^{k+1}}{2}.
 \end{equation}
 For $k=1$, the first equality holds. 
 \item[(iii)] $\beta_3({\mathcal F}/{\mathbb F}_q)\geq \frac{2(q^2-1)}{3(q+2)}$.
\end{itemize}
\end{proposition}

\begin{proof}
Let $H_0=\mathbb{F}_{q^3}(x_0)$ be the rational function field and  $H_1={\mathbb F}_{q^3}(x_0,x_1)$ 
where $f(x_0,x_1)=0$. We know the following from  \cite[Theorems 2.2, 3.4 and 6.5]{bagast}:
\begin{itemize}
\item [(a)]  For all $k\geq 0$, the extensions $H_{k+1}/H_{k}$ are Galois with $[H_1:H_0]=q(q-1)$ 
and $[H_{k+1}:H_{k}]=q$ for  all $k\geq 1$. 
\item [(b)] The descent sequence of $\mathcal{H}/{\mathbb F}_{q^3}$ to any field $K$ containing ${\mathbb F}_q$ is a tower. 
\item [(c)]  The genus $g(H_k)\leq \frac{(q^3+q^2-2q)q^k}{2}$ for all $k\geq 0$.
\item [(d)] $\operatorname{Split}(\mathcal{H})\supseteq \left\{(x_0=\alpha)| 
\alpha\in \mathbb{F}_{q^3}\setminus \mathbb{F}_q  \right\}$. 
\end{itemize}
(i)  Since
\[B_1(H_0)=B_1(F_0)+3 B_3(F_0),\]
we have  that $B_3(F_0)=\frac{q^3-q}{3}$. Obviously, any degree $3$ 
place $P\in \mathbb{P}(F_0)$ splits completely in $H_0$ and its extensions 
are in the set $Split(\mathcal{H})$. Hence, clearly,  $P$ splits 
completely in $F_k$ for all $k\geq 0$. Thus, for all $k\geq 0$, we obtain that 
\[B_3(F_k)\geq B_3(F_0)[F_k:F_0],\] 
and hence (i) follows.\\
(ii)  We first note that after taking  the constant field 
extension the genera do not change, see \cite[Theorem 3.6.3(b)]{stic2}. 
The second inequality follows from (b) as $g(F_k)=g(H_k)$.  
To prove the first inequality, we consider the sequence 
${\mathcal F}'/{\mathbb F}_{q^2}=(F'_k/{\mathbb F}_{q^2})_{k\geq 0}$, where $F'_k=F_k{\mathbb F}_{q^2}$. 
Then $g(F_k)=g(F'_k)$ for all $k\geq 0$.  We first compute $g(F'_1)$. 
Let $Q$ be a place of $F'_1$ and $P=Q\cap F'_0$. 
By using \cite[Lemma 5.2]{bagast} and  \cite[Definition 7.4.12 and Theorem 3.5.1]{stic2}, we obtain that
\begin{equation*}
d(Q|P)=\begin{cases} 0      \textrm{ if } P=(x_0=0)\\
                     2q-2   \textrm{ if } P=(x_0=\infty) \textrm{ or } (x_0=\alpha) \textrm{ with } \alpha \in {\mathbb F}_q^*\\                  
                     q-2      \textrm{ if } P=(x_0=\beta) \textrm{ with } \beta \in {\mathbb F}_{q^2}\setminus {\mathbb F}_q. 
                     \end{cases}
\end{equation*}
Let $s=\#\{ Q\in \mathbb{P}(F'_1): Q|P\}$ and $f=f(Q|P)$ (note that $F'_1/F'_0$ is Galois).  By the same lemma,
\begin{equation}\label{sf}
sf= \begin{cases}
              q-1 \textrm{ if }  P=(x_0=\infty) \textrm{ or } (x_0=\alpha) \textrm{ with } \alpha \in {\mathbb F}_q^*\\
              q      \textrm{ if } P=(x_0=\beta) \textrm{ with } \beta \in {\mathbb F}_{q^2}\setminus {\mathbb F}_q. 
              \end{cases}
\end{equation}
Hence, 
\begin{eqnarray*}
\deg \operatorname{Diff}(F'_1/F'_0)&=&\sum_{P}\sum _Q d(Q|P)\deg Q\\
                   &=&(2q-2)(q-1)+(2q-2)(q-1)^2+q(q-2)(q^2-q)\\
                   &=&q^4-q^3-2q^2+2q.
\end{eqnarray*}
The Hurwitz Genus Formula \cite{stic2} for the extension $F'_1/F'_0$ now yields that 
\begin{equation}\label{g(F_1)}
g(F'_1)=\frac{1}{2}(q^4-q^3-4q^2+4q+2).
\end{equation}
Next, we estimate $g(F'_k)$ for any $k\ge 2$. Let $Q_k$ be a place of $F'_k$ and $Q=Q_k\cap F'_1$. 
Suppose that $Q$ lies above one of the following places:
\begin{itemize}
\item $(x_0=\infty)$ or $(x_0=\alpha)$  with $\alpha \in {\mathbb F}_q^*$, 
\item  $(x_0=\beta)$  with  $\beta \in {\mathbb F}_{q^2}\setminus {\mathbb F}_q$ and $(x_1=\alpha)$ with $\alpha \in {\mathbb F}_q^*$.
\end{itemize} 
 By \cite[Lemmas 5.2, 6.3, 6.4]{bagast} (also c.f.proof of \cite[Theorem 6.5]{bagast}), we have that
\begin{equation}\label{ed}
e(Q_k|Q)=q^{k-1}\quad \textrm{ and } \quad d(Q_k|Q)=2(e(Q_k|Q)-1)=2(q^{k-1}-1).
\end{equation}
Hence, using (\ref{sf}) and (\ref{ed}) yields that
\begin{eqnarray*}
\deg \operatorname{Diff}(F'_k/F'_1)&\geq&  \sum_Q\sum_{Q_k} d(Q_k|Q)\deg Q_k\\
                                 &=& 2(q^{k-1}-1)[(q-1)+(q-1)^2+(q^2-q)(q-1)]\\
                                 &=& 2(q^3-q^2)(q^{k-1}-1).
\end{eqnarray*}
Then using (\ref{g(F_1)}) and the Hurwitz Genus Formula \cite{stic2} for the extension $F'_k/F'_1$ gives that
\begin{equation}\label{g(F_k)}
g(F_k)=g(F'_k)\geq \frac{1}{2}(q^4+q^3-6q^2+4q)q^{k-1}-q^3+q^2+1.
\end{equation}
Note that when $k=1$, the equality in (\ref{g(F_k)}) holds.\\
(iii) follows from (i) and (ii).\\
\end{proof}

\begin{proposition}\label{hparam2}
The steps of the towers 
$\mathcal{H}/{\mathbb F}_{q^3}=(H_k/{\mathbb F}_{q^3})_{k\geq 0}$ and $\mathcal{F}/{\mathbb F}_q=(F_k/{\mathbb F}_q)_{k\geq 0}$ 
defined in Proposition \ref{param2} have a class number such that
$$h(H_k/{\mathbb F}_{q^3})\geq h_{BRT1,k}$$ 
and $$h(F_k/{\mathbb F}_q)\geq h_{BRT2,k}$$
where
\begin{equation*}
\begin{split}
h_{BRT1,k}= 
\frac{(q^3-1)^2}{(g_k+1)(q^3+1)-B_{1,k}} 
\bigg(
 \left(\begin{array}{c}
 B_{1,k}+g_k-1\\
B_{1,k}
\end{array}\right)
 \\
 + 
q^{3(g_k-1)}\left [ \left( \frac{q^3}{q^3-1}\right)^{B_{1,k}}
-  \right. \\ \left. B_{1,k}\left(
\begin{array}{c}
 B_{1,k}+g_k-2\\
 B_{1,k}
\end{array}
\right) \int_{0}^{\frac{1}{q^3}} \frac{(\frac{1}{q^r}-t)^{g_k-2}}{(1-t)^{B_{1,k}+g_k-1}}\rm dt. \right ]
\biggr)
\end{split}
\end{equation*}
and
 \begin{equation*}
\begin{split}
h_{BRT2,k}=
\frac{(q-1)^2}{(g_k+1)(q+1)} 
\bigg(
 \left(\begin{array}{c}
 B_{3,k}+\lfloor \frac{g_k-1}{3} \rfloor\\
B_{3,k}
\end{array}\right)
 \\
+ 
q^{g_k-1}  \left [ \left( \frac{q^3}{q^3-1}\right)^{B_{3,k}} 
-\right .\\ \left .
B_{3,k}\left(
\begin{array}{c}
 B_{3,k}+\lfloor \frac{g_k-2}{3}\rfloor\\
 B_{3,k}
\end{array}
\right) \int_{0}^{\frac{1}{q^3}} \frac{(\frac{1}{q^3}-t)^{\lfloor \frac{g_k-2}{3} \rfloor}}{(1-t)^{B_{3,k}
+\lfloor \frac{g_k-2}{3} \rfloor+1}}\rm dt. \right ]
\biggr)
\end{split}
\end{equation*}
with $B_{3,k}=\frac{(q^3-q^2-q+1)q^{k+1}}{3}$, $B_{1,k}=(q^3-q^2-q+1)q^{k+1}$ and 
$g_k=  \frac{1}{2}\big(q^4+q^3-6q^2+4q)q^{k-1}-q^3+q^2+1$ for all $k\geq 1$
\end{proposition}

\begin{proof}
The bound $h_{BRT1,k}$ follows directly from Theorem \ref{mainmain} by using the lower bound on the genus given 
by Proposition \ref{param2}  since the functions $h_{BRT1,k}$ and $h_{BRT2,k}$ with respect to 
the genus are increasing. Moreover, we apply 
Theorem \ref{mainmain} by using the lower bound on the number of places of degree one given by Proposition \ref{param2} 
with the parameters $l_1=g_k-1$ and $m_1=g_k-2$. The bound $h_{BRT2,k}$ follows directly from Theorem \ref{mainmain} 
as in the previous case by using the lower bound on the number of places of degree $3$ given by Proposition \ref{param2} 
with the parameters $l_3=\lfloor \frac{g_k-1}{3}\rfloor$ and $m_3=\lfloor \frac{g_k-2}{3}\rfloor$. 
\end{proof}

\medskip

{\bf Numerical estimations:} 

\medskip

Here are some effective numerical estimations of  the class numbers of some steps of the towers 
$\mathcal{H}/{\mathbb F}_{q^3}$ and $\mathcal{F}/{\mathbb F}_q$ given by Proposition \ref{hparam2} with $q=2$.

\medskip

\begin{center}
\begin{tabular}{|l|c|c|c|c|c|r|}
\hline
$q^3$  & step k & $g_k$ &  $B_{1,k}$  & $l_1$ & $m_1$ & $h_{BRT1,k}$   \\
\hline
8 &  2  &    5 &     24  &  4  &   3 &   125 537 \\
 \hline
8 &    3   &   13   &   48  &   12  &  11 &   $2.556 \times 10^{13}$  \\
 \hline     
8 &      4    &    29 &      96  &   28  &    27 &   $2.010\times 10^{30}$ \\
 \hline      
\end{tabular}
\end{center}

\medskip

\begin{center}
\begin{tabular}{|l|c|c|c|c|c|r|}
\hline
q  & step k & $g_k$ &  $B_{3,k}$  & $l_3$ & $m_3$ & $h_{BRT2,k}$   \\
\hline 
2 &  2  &    5 &     8  &  1  &   1 &   3 \\
 \hline
2 &    3   &   13   &   16  &   4  &  3 &   771  \\
 \hline     
2 &      4    &  29  &      32  &   9  &    9 &   212 127 395 \\
 \hline      
\end{tabular}
\end{center}

\medskip

\subsection{General case $N=\{r_i\}_{i}$}

In this section, we first recall a method for constructing towers of function fields  ${\mathcal F}/{\mathbb F}_q$ with various invariants 
$\beta_r({\mathcal F})$ being positive. This method is based upon the {\it compositum} of asymptotically good towers with different 
suitable  extensions. Then we give some examples of explicit towers having several positive invariants $\beta_{r_i}$, 
for which we describe precisely several parameters (such as genus, rational places, etc.).  
Then we give an estimation of class number of each step of these towers. 

Let us recall a useful lemma concerning the genera in towers (See \cite[Theorem 3.6]{gastth}) 
which we will use often in the subsequent sections.

\begin{lemme} \label{lemGenus}
              Let ${\mathcal F}/{\mathbb F}_q=(F_k/{\mathbb F}_q)_{k\geq 0}$ be a tower and  $E/F_0$ be a finite separable 
extension such that $R({\mathcal F})$ is finite and all $P\in R({\mathcal F})$ are tame in $E$. 
Suppose that ${\mathcal E}/{\mathbb F}_q=(E_k/{\mathbb F}_q)_{k\geq 0}$, with $E_k=EF_k$, is a composite tower. Set $m=[E:F]$. 
Then for all $k\geq 0$, 
\begin{equation*} \label{genusGS}
            g(E_k)=\bigg(g(E)-mg(F)+m-\frac{s+2}{2}\bigg) [E_k:E]+m(g(F_k)-1)+\frac{r(k)}{2}+1,
\end{equation*}
where
\begin{eqnarray*} \label{deltaGS}
            s&=&\sum_{\substack{Q\in \mathbb{P}(E)\\ Q\cap F\in R({\mathcal F})}} d(Q|Q\cap F)\deg Q\quad  \textrm{ and }\\ 
              r(k)&=&\sum_{P\in R({\mathcal F})}\sum_{\substack{P_k\in \mathbb{P}(F_k)\\ 
P_k|P}}\deg P_k \sum_{\substack{\;Q_k\in \mathbb{P}(E_k)\\ Q_k|P_k}}(e(Q_k|P_k)-1)f(Q_k|P_k).
\end{eqnarray*}
\end{lemme}

\subsubsection{Family  1}
We begin with an example given in \cite{hesttu}.

\begin{proposition}\label{lemGS}
             Let $N\subset {\mathbb N}$ be a finite set, $m=\sum_{i\in N}i$, and $q$ be a prime power with $(m,q)=1$. 
Consider the tower ${\mathcal F}/{\mathbb F}_{q^2}$ given in Proposition \ref{optimal.exm}.
              Let $E=F_0(z)$ with $z$ a root of the polynomial 
\begin{equation} \label{definingGS}
                  \varphi(T)=\prod_{i\in N}g_i(T)-x_0^{q^2}+x_0 \in F_0[T],
\end{equation}
             where $g_i(T)\in {\mathbb F}_{q^2}[T]$ is a monic, irreducible polynomial of $\deg g_i(T)=i$, 
for each $i\in N$. Then the sequence ${\mathcal E}/{\mathbb F}_{q^2}=(E_k/{\mathbb F}_{q^2})_{k\geq 0}$, with $E_k=EF_k$, 
is a composite tower such that
  \begin{itemize}
 \item[(i)] $B_1(E_k)\geq 1$ for all $k\geq 1$ and $B_i(E_k)\geq q^k(q^2-1)$ for all $i\in N$ and all $k\geq 0$
 \item[(ii)] for all $k\geq 2$, one has
 $$g(E_k)=\left\{
\begin{array}{ll}
\frac{1}{2}(m(q^2+2q+2)-q^2)q^k-m(q+2)q^{\frac{k}{2}}+
\frac{m+1}{2} &\hbox{if }  k \hbox{ even};\\ 
\frac{1}{2}(m(q^2+2q+2)-q^2)q^k-\frac{m}{2}(q^2+3q+1)q^{\frac{k-1}{2}}+\frac{m+1}{2} &\hbox{if } k \hbox{ odd}.
\end{array}
\right .
$$

\item[(iii)]  \[ \beta_i({\mathcal E})=\frac{2(q^2-1)}{m(q^2+2q+2)-q^2}
 \textrm{ for all $i\in N$ and  $\beta_i({\mathcal E})=0$ for all $i\notin N$.}\] 
\end{itemize}
\end{proposition}
\begin{proof}  The assertions (i) and (iii) are given in  \cite[Example 3.15]{hesttu}. 
We just need to prove (ii). For this we use Lemma \ref{lemGenus}.
It is given in \cite[Example 3.15]{hesttu} that the genus
\begin{equation}\label{gE}
g(E)=\frac{(m-1)(q^2-1)}{2}.
\end{equation}
Moreover, in that example it is shown that $P_0$ is unramified in $E$ and $P_\infty$ is totally 
ramified in $E$. We know from \cite{gast} that 
 \[R({\mathcal F})=\{P_0,P_\infty\}\quad \textrm{where $P_0$ (resp. $P_\infty$) is the zero (resp. the pole) of $x_0$ in $F_0$}.\]
Hence, by using Abhyankar's Lemma and Dedekind's Different Theorem \cite{stic2}, 
we obtain that the $s$ and $r(k)$ defined in Lemma \ref{lemGenus} are as follows: 
\begin{equation} \label{sr}
            s=r(k)=m-1 \quad \textrm{ for all $k\geq 0$.}
\end{equation}
Now applying Lemmas \ref{lemGenus} and \ref{lemGS}(iii) with $r=2$, (\ref{gE}) and (\ref{sr})  give the desired result.
\end{proof}

\begin{proposition}\label{hparam3}
The steps of the tower 
$\mathcal{E}/{\mathbb F}_{q^2}=(E_k/{\mathbb F}_{q^2})_{k\geq 0}$ 
defined in Proposition \ref{lemGS} have a class number such that
$$h(E_k/{\mathbb F}_{q^2})\geq h_{BRT1,k}$$ 
where
\begin{equation*}
\begin{split}
h_{BRT1,k}= 
\frac{(q^2-1)^2}{(g(E_k)+1)(q^2+1)-1} 
\bigg(
\prod_{i\in N\cup \{1\}} 
 \left(\begin{array}{c}
 B_{i,k}+m_i\\
B_{i,k}
\end{array}\right)
 \\
 + 
q^{2(g(E_k)-1)} \prod_{i\in N \cup \{1\} } \left [ \left( \frac{q^{2i}}{q^{2i}-1}\right)^{B_{i,k}}
-  \right. \\ \left. B_{i,k}\left(
\begin{array}{c}
 B_{i,k}+m_i\\
 B_{i,k}
\end{array}
\right) \int_{0}^{\frac{1}{q^{2i}}} \frac{(\frac{1}{q^{2i}}-t)^{m_i}}{(1-t)^{B_{i,k}+m_i+1}}\rm dt. \right ]
\biggr)
\end{split}
\end{equation*}
 where $B_{1,k}=1$, for all $1\neq i\in N$, $B_{i,k}= q^k(q^2-1)$ and $\sum_{i\in N\cup\{1\}} im_i\leq g(E_k)-2$ for all $k\geq 1$.
\end{proposition}

\begin{proof}
The bound $h_{BRT1,k}$ follows directly from Theorem \ref{mainmain} by taking 
$l_i=m_i$ for $i\in N\cup \{ 1\}$ and the lower bounds for the number of places of each degree given by 
Proposition \ref{lemGS}.
\end{proof}

\medskip

{\bf Numerical estimations:} 

\medskip

Here are some effective numerical estimations of  the class numbers of some steps $E_k/{\mathbb F}_{q^2}$ of the tower 
$\mathcal{E}/{\mathbb F}_{q^2}$ given by Proposition \ref{lemGS} with $q=2$ and $N=\{2,3\}$. 
These estimations are given by using Proposition \ref{hparam3}.
 
\medskip

\begin{center}
\begin{tabular}{|l|c|c|c|c|c|c|c|c|r|}
\hline
$q^2$  & step k & $g(E_{k})$ &  $B_{1,k}$  & $B_{2,k}$ & $B_{3,k}$ & $m_1$& $m_2$ & $m_3$ & $h_{BRT1,k}$   \\
\hline
4 &  2  &    55 &     1  &  12  &   12 &  13 & 11 & 6 &  3.65792120927 $\times 10^{31}$\\
 \hline
4 &    3   &   132   &   1  &   24  &  24 &  29 & 25 & 17 &  9.19830033703 $\times 10^{77}$ \\
 \hline     
 \end{tabular}
\end{center}

\medskip

\subsubsection{Family  2}

This  family of towers is based upon the following result:
\begin{proposition} \label{construction} Let $F/{\mathbb F}_q$ be an algebraic function field with a 
finite set of places $S$ and a finite separable extension $F'/{\mathbb F}_q$. Further let $N\subset{\mathbb N}$ be a finite set with 
                   $m=\sum_{i\in N}i$. Suppose that $F/{\mathbb F}_q$ has a rational place $Q$ 
which has a rational extension $Q'$ in $F'$ such that $(m,e(Q'|Q))=1$.    
                   Define     
                   $E=F(z)$ where $z$ is a root of the polynomial 
                  \[\varphi(T)=\prod_{i\in N} p_i(T)-\alpha \in F[T]\]
                which has the following properties:
 \begin{itemize}
           \item[(a)] for all $P\in S$, each $p_i(T)$ is an irreducible monic polynomial over $k(P)$ of $\deg p_i(T)=i$.
           \item[(b)] $\alpha \in F$ and $\alpha(P)=0$ for all $P\in S$,
           \item[(c)] $v_Q(\alpha)<0$ and $(v_Q(\alpha),m)=1$.
 \end{itemize}
             Set $E'=EF'$. Then $E'/F$ is a separable extension such that 
 \begin{itemize}
            \item[(i)] $Q$ is totally ramified in $E$, $[E:F]=[E':F']=m$ and ${\mathbb F}_q$ is algebraically closed in $E'$.  
            \item[(ii)] Each place $P\in S$ has exactly one extension $Q_i \in E$ with 
            \[ \deg Q_i=i\deg P \textrm{ for all $i\in N$.}\]
            \item[(iii)] If $P$ splits completely  in $F'$, then each extension of $P$ in $E$ splits completely in $E'$.
 \end{itemize}
\end{proposition} 
 \begin{proof} (i) By applying the generalized Eisenstein's Irreducibility Criterion \cite{stic2} 
with the place $Q$ and using (c), we obtain that   
                       $\varphi(T)$ is irreducible over $F$ and $Q$ is totally ramified in $E$. 
Since $(m,e(Q'|Q))=1$ for some rational place $Q'$ of $F'$    
                       lying over $Q$, it follows from Abhyankar's Lemma \cite{stic2} that $Q'$ 
is totally ramified in $E'=EF'$. Thus, (i) follows.\\
              (ii) The proof is clear by Kummer's Theorem \cite{stic2}, and the properties (a) and (b). \\
              (iii) See \cite[Proposition 3.9.6(a)]{stic2}.
\end{proof}

We note that a general version of Proposition \ref{construction} is given in \cite{hesttu}. 
\begin{remarque} \label{rem.const} In Proposition \ref{construction},  
the elements in the set $N$ do not need to be distinct if the following holds: 
for all $P\in S$ and  each  $i\in N$, the polynomials $p_i(T)\in \mathcal{O}_P[T]$ 
are pairwise distinct  and irreducible over $k(P)$ with $\deg p(T)=i$.
\end{remarque}

\medskip
   
\noindent {\bf a) Example 1}:

\begin{proposition}\label{quadratic1} 
Consider the  tower ${\mathcal F}/{\mathbb F}_q=(F_k/{\mathbb F}_q)_{k\geq 0}$ 
given in Proposition \ref{param2} with $q=2$ and the rational function field $F_0={\mathbb F}_q(x_0)$. 
Let $E=F_0(z)$ where $z$ is a root of the polynomial 
\[\varphi(T)=T^4+T+x_0^7+1 \in F_0[T].\]
Then the sequence ${\mathcal E}/{\mathbb F}_2=(E_k/{\mathbb F}_2)_{k\geq 0}$, with $E_k=EF_k$, is a composite 
quadratic tower with the following properties:
\begin{itemize}
\item [(i)] $B_3(E_k)\geq 2^{k+2}$ and $B_6(E_k)\geq 2^{k+1}$ for all $k\geq 0$,
\item[(ii)] $2^{k+3}-15\leq g(E_k)\leq 7.2^{k+2}-3$ for all $k\geq 0$,
\item[(iii)] $\beta_3({\mathcal E})\geq \frac{1}{7}$ and $\beta_6({\mathcal E})\geq \frac{1}{14}$.
\end{itemize}
\end{proposition}

\begin{proof}  We first  need to show that ${\mathcal E}/{\mathbb F}_2$ is a quadratic tower.
We have that 
\[\varphi(T)=T(T+1)(T^2+T+1)+(x_0+1)(x_0^3+x_0^2+1)(x_0^3+x_0+1) \in F_0[T].\]
Let $k\geq 1$. We apply Proposition \ref{construction} and Remark \ref{rem.const}  
with $F,F'=F_k, N=\{1,1,2\}, Q=P_\infty=(x_0=\infty)$, and the set 
$S=\{Q_1,Q_2\} \in \mathbb{P}(F_0)$ with $Q_1:=(x_0^3+x_0^2+1=0)$ and  $Q_2:=(x_0^3+x_0+1=0)$. 
It is shown in the proof of Proposition \ref{param2}(i) that the places $Q_1$ and $Q_2$ split 
completely in $E_k$ for all $k\geq 0$. Hence, Proposition \ref{construction}(ii) and (iii) hold, 
and  $[E:F_0]=\deg \varphi(T)=4$. 

We now claim that $E/F_0$ and $F_k/F_0$ are linearly disjoint and ${\mathbb F}_2$ is algebraically 
closed in $E_k=EF_k$. Note that $[F_{k+1}:F_k]=2$ for all $k\geq 0$.
By \cite[Lemma 5.2]{bagast}, the place $P=(x_0=1)$ of $F_0$  is totally ramified  in $F_k$. 
By Kummer's Theorem \cite{stic2} $P$ is unramified in $E$ and it has an extension, say $P'$, of degree one. 
Let $Q'\in \mathbb{P}(E_k)$ lying above $P'$. By Abhyankar's Lemma $e(Q'|P')=2^k=[F_k:F_0]\geq [E_k:E]$. 
Since $e(Q'|P')\leq [E_k:E]$, we have the equality, and so $P'$ is totally ramified in $E_k$. 
Now the claim follows. Hence, ${\mathcal E}/{\mathbb F}_2$ is a tower and since ${\mathcal F}/{\mathbb F}_2$ is quadratic, it is also quadratic. \\
(i)  Let $n_i=\# \{R\in \mathbb{P}(E): R|Q_1 \textrm{ or } R|Q_2 \textrm{ with } \deg R=i\}$ 
for $i\geq 1$. By using Proposition \ref{construction}(ii),(iii), we obtain that
\[B_3(E_k)\geq n_3[E_k:E]\geq 4[E_k:E]=2^{k+2} \quad \textrm{ and } \]
\[B_6(E_k)\geq n_6[E_k:E]\geq 2[E_k:E]=2^{k+1}.\qquad\quad\] 
(ii)  Since the genera do not 
change after taking the constant field extensions, we can consider 
${\mathcal E}'/{\mathbb F}_4=(E'_k/{\mathbb F}_4)_{k\geq 0}$, with $E'_k:=E_k{\mathbb F}_4$, and ${\mathcal F}'/{\mathbb F}_4 =(F'_k/{\mathbb F}_4)_{k\geq 0}$, 
with $F'_k:= F_k{\mathbb F}_4$. 
By \cite[Proposition 3.7.10]{stic2}, $E'/F'_0$ is an elementary Abelian extension of degree $4$. 
Moreover, the only ramified place in $E'/F'_0$ is the pole of $x_0$, say $P_\infty$, with 
\begin{equation}\label{abelian.ext}
e(Q|P_\infty)=4 \quad \textrm{ and }\quad  d(Q|P_\infty)=(4-1)(-v_{P_\infty}(x_0^7+1)+1)=24
\end{equation}
where  $Q\in \mathbb{P}(E')$ lies above $P_\infty$.
Hence, by the Hurwitz  Genus Formula \cite{stic2} 
\begin{equation}\label{genusE}
g(E')=9.
\end{equation}
By \cite[Lemma 3.11]{hesttu}, we have that 
\begin{eqnarray}\label{upper}
         g(E_k)&\leq &[E:F_0]g(F_k)+[F_k:F_0]g(E)-[E:F_0][F_k:F_0]g(F_0)\\
               &+&([E:F_0]-1)([F_k:F_0]-1)\\
               &=&4g(F_k)+12\cdot 2^k-3\quad \textrm{ by (\ref{genusE}) as $g(E)=g(E')$.}
\end{eqnarray}
Moreover, one can easily conclude  from the Hurwitz Genus Formula \cite{stic} for the extension $E_k/F_k$ that 
\begin{equation}\label{low}
g(E_k)\geq 4g(F_k)-3.
\end{equation}
By applying Proposition \ref{param2} (ii) with $q=2$, we obtain that
\begin{equation}\label{genusFk}
2^{k+1}-3\leq g(F_k)\leq 2^{k+2}.
\end{equation}
Now combining (\ref{genusE}), (\ref{upper}), (\ref{low}) and (\ref{genusFk}) yields the desired bounds in (ii). \\
(iii) follows from (i) and (ii).
\end{proof}

\begin{corollaire} \label{coroquadratic1}
Let ${\mathcal E}/{\mathbb F}_2$ be the tower given in Proposition \ref{quadratic1} 
and  ${\mathcal E}'/{\mathbb F}_8=(E'_k/{\mathbb F}_8)_{k\geq 0}$ with $E'_k=E_k{\mathbb F}_8$. Then ${\mathcal E}'/{\mathbb F}_8$ is a 
quadratic tower with the following properties:
\begin{itemize}
\item[(i)] $B_1(E'_k)\geq 3\cdot 2^{k+2}$ and $B_2(E'_k)\geq 3\cdot 2^{k+1}$ for all $k\geq 0$. 
\item[(ii)] $\beta_1({\mathcal E}')\geq \frac{3}{7}$ and $\beta_2({\mathcal E}')\geq \frac{3}{14}$. 
\end{itemize}
\end{corollaire}
\begin{proof} The proof follows from \cite[Proposition 5.1.9]{stic2} and Proposition \ref{quadratic1}.
\end{proof}

\begin{proposition}\label{hparam4}
The steps of the towers 
$\mathcal{E}/{\mathbb F}_{2}=(E_k/{\mathbb F}_{2})_{k\geq 0}$ and $\mathcal{E}'/{\mathbb F}_{8}=(E^{'}_k/{\mathbb F}_{8})_{k\geq 0}$ 
respectively defined in Proposition \ref{quadratic1} and Corollary \ref{coroquadratic1} have a class number such that
$$h(E_k/{\mathbb F}_{2})\geq h_{BRT1,k}$$ 
and
$$h(E^{'}_k/{\mathbb F}_{8})\geq h_{BRT2,k}$$ 
where
\begin{equation*}
\begin{split}
h_{BRT1,k}= 
\frac{1}{3(8.2^k-14)} 
\bigg(
\prod_{i\in \{3,6\}} 
 \left(\begin{array}{c}
 B_{i,k}+m_i\\
B_{i,k}
\end{array}\right)
 \\
 + 
2^{(8.2^k-16)} \prod_{i\in \{3,6\} } \left [ \left( \frac{2^{i}}{2^{i}-1}\right)^{B_{i,k}}
-  \right. \\ \left. B_{i,k}\left(
\begin{array}{c}
 B_{i,k}+m_i\\
 B_{i,k}
\end{array}
\right) \int_{0}^{\frac{1}{2^i}} \frac{(\frac{1}{2^i}-t)^{m_i}}{(1-t)^{B_{i,k}+m_i+1}}\rm dt. \right ]
\biggr)
\end{split}
\end{equation*}

and

\begin{equation*}
\begin{split}
h_{BRT2,k}= 
\frac{49}{9(8.2^k-14)-B_{1,k}} 
\bigg(
\prod_{i\in \{1,2\}} 
 \left(\begin{array}{c}
 B_{i,k}+m_i\\
B_{i,k}
\end{array}\right)
 \\
 + 
2^{(8.2^k-16)} \prod_{i\in \{1,2\} } \left [ \left( \frac{8^{i}}{8^{i}-1}\right)^{B_{i,k}}
-  \right. \\ \left. B_{i,k}\left(
\begin{array}{c}
 B_{i,k}+m_i\\
 B_{i,k}
\end{array}
\right) \int_{0}^{\frac{1}{8^i}} \frac{(\frac{1}{8^i}-t)^{m_i}}{(1-t)^{B_{i,k}+m_i+1}}\rm dt. \right ]
\biggr)
\end{split}
\end{equation*}

 with $$B_{3,k}= 2^{k+2} \mbox{, } B_{6,k}= 2^{k+1}\mbox{, }B_{1,k}=3.2^{k+2}\mbox{, } B_{2,k}=3.2^{k+1}$$ and 
 $$3m_3+6m_6 \leq 8.2^k-17$$ and $$m_1+2m_2 \leq 8.2^k-17$$ for all $k\geq 1$.
\end{proposition}

\begin{proof}

The bound $h_{BRT1,k}$ follows directly from Theorem \ref{mainmain} by taking 
$l_i=m_i$ for $i=3,6$ and the lower bounds for the number of places of each degree given 
by Proposition \ref{quadratic1}.  The bound $h_{BRT2,k}$ follows directly from Theorem \ref{mainmain} by taking 
$l_i=m_i$ for $i=1,2$ and the lower bounds for the number of places of each degree given 
by Corollary \ref{coroquadratic1}.
\end{proof}

\medskip

{\bf Numerical estimations:} 

\medskip

Here are some effective numerical estimations of  the class numbers of some steps $E_k/{\mathbb F}_{2}$ and  $E^{'}_k/{\mathbb F}_{8}$ of 
respectively the towers $\mathcal{E}/{\mathbb F}_{2}$ and $\mathcal{E}'/{\mathbb F}_{8}$ given by Proposition \ref{quadratic1}. 
These estimations are given by using Proposition \ref{hparam4}.
 
\medskip

\begin{center}
\begin{tabular}{|l|c|c|c|c|c|c|c|c|r|}
\hline
q  & step k & $g_k$ &  $B_{3,k}$  & $B_{6,k}$ & $m_3$ & $m_6$ & $h_{BRT1,k}$   \\
\hline
2 &  2  &    17 &     16  &  8  &   5&  0 &    10 254\\
 \hline
2 &    3   &   49   &   32  &   16  &  11 &  2 &    1.71832288189 $\times 10^{14}$ \\
 \hline     
  \end{tabular}             
  \end{center}  

\medskip         

\begin{center}
\begin{tabular}{|l|c|c|c|c|c|c|c|c|r|}
\hline
$q^3$  & step k & $g_k$ &  $B_{1,k}$  & $B_{2,k}$ & $m_1$ & $m_2$ & $h_{BRT2,k}$   \\
\hline
8 &  2  &    17 &     48 &  24  &   11&   2& 1.0021797677 $\times 10^{17}$   \\
 \hline
8 &    3   &   49   &   96  &   48  &  33 &   7& 2.42674161125  $\times 10^{48}$  \\
 \hline
 \end{tabular}      
\end{center}       

\medskip
   
\noindent {\bf b) Example 2}: 

\begin{proposition} \label{param1bis}
Let $N\subset {\mathbb N}$ be a finite set and set $m=\sum_{i\in N}i$. 
Let $l$ be a prime power, $q=l^r$ with $r\geq 2$, and  $d=\frac{q-1}{l-1}$.  Suppose that $(m,dq)=1$. 
Further let ${\mathcal F}/{\mathbb F}_q=(F_k/{\mathbb F}_q)_{k\geq 0}$, with the rational function field  $F_0={\mathbb F}_q(x_0)$, 
be the tower which is recursively  defined by the polynomial 
 \begin{equation}\label{wulftange.eq1}
 f(X,Y)=Y^d-a(X+b)^d+c \in {\mathbb F}_q[X,Y] 
 \end{equation} 
 with  $a,c\in {\mathbb F}_l, b\in {\mathbb F}^*_q$ and  $ab^d+c=0$.  
Let $E=F_0(z)$ with $z$ a root of the polynomial 
\begin{equation}\label{definingE}
\varphi(T)=\prod_{i\in N}p_i(T)-\frac{1}{x_0} \in F_0[T],
\end{equation}
where each $p_i(T)\in {\mathbb F}_q[T]$ is an irreducible monic polynomial of $\deg p_i(T)=i$.
Suppose  that $(\varphi(T),\varphi'(T))=1$ at the places $P\in R({\mathcal F})\setminus \{P_0\}$ and $\varphi(T)$ is separable.
Then the sequence ${\mathcal E}/\mathbb{F}_{q}=(E_k/{\mathbb F}_{q})_{k\geq 0}$, with  $E_k:=EF_k$, is a composite tower  such that
\begin{itemize}
 \item[(i)] $B_1(E_k)\geq 1$ and $B_i(E_k)\geq d^k$ for all $i\in N$ and  $k\geq 0$,
 \item[(ii)] for all $k\geq 1$
 \[\bigg(\frac{mq-m-2}{4}\bigg)d^k+\frac{1}{2} \leq g(E_k)\leq \bigg(\frac{mq-m-1}{2}\bigg)d^k+\frac{(1-q)m+1}{2}.\]

\item[(iii)]  $\beta_i({\mathcal E})\geq\frac{2}{mq-m-1}$ for all $i\in N$.
\end{itemize}
\end{proposition}
\begin{proof} By \cite[Theorem 4.2.3 and Lemma 4.2.2]{wulf}, the tower ${\mathcal F}/{\mathbb F}_q$ has the following properties: 
 \begin{itemize}
 \item[(a)] the pole (resp. the zero) of $x_0$, say $P_\infty$ (resp.  $P_0$) 
splits completely  (resp. totally ramified) in ${\mathcal F}/{\mathbb F}_q$,
 \item[(b)] $\beta_1({\mathcal F})=\frac{2}{q-2}$,
 \item[(c)] for any $k\geq 1$, the genus
 \[g(F_k)=\bigg(\frac{q-2}{2}\bigg)d^k-\frac{1}{2}a_k +1\textrm{ with } \lim_{k\to \infty} \frac{a_k}{[F_k:F_0]}=0, \textrm{ where }\]
 \[a_k:=\sum_{P\in R({\mathcal F})}\sum_{\substack{Q\in \mathbb{P}(F_k)\\ Q|P}} \deg Q.\]
  \item[(d)] $R({\mathcal F})=\{(x_0-\alpha)\in \mathbb{P}(F_0)|  \; \alpha \in {\mathbb F}_q\}$.
\end{itemize}
By applying Proposition \ref{construction} with the set $S=\{P_\infty\}$ and $Q=P_0$, 
we obtain that the sequence ${\mathcal E}/{\mathbb F}_q=(E_k/{\mathbb F}_q)_{k\geq 0}$ is a tower with the following properties:
 \begin{itemize}
 \item[(*)] $[E:F_0]=[E_k:F_k]=\deg \varphi(T)=m$ for all $k\geq 0$, 
(observe that then $E/F$ and $F_k/F$ are linearly disjoint for all $k\geq 0$ ).
 \item[(**)]for each $i\in N$, the place $P_\infty$ has exactly one extension $Q_i$ in $E$ with  $\deg Q_i=i$
 and $Q_i$ splits completely in ${\mathcal E}$. 
\end{itemize}
By (a), Proposition \ref{construction}(i), and Abhyankar's Lemma\cite{stic2}, 
$P_0$ is totally ramified in $E_k$ for all $k\geq 0$. Hence, $E_k$ has an ${\mathbb F}_q$-rational place, 
namely the extension of $P_0$. Thus, $B_1(E_k)\geq 1$. The second part of (i) is clear by (*) and (**).\\
(ii) We will use Lemma \ref{lemGenus}. It is clear that $g(E)=0$. 
Let $P\in \mathcal{R}({\mathcal F})\setminus\{P_0\}$. By assumption $(\varphi(T),\varphi'(T))=1$ at $P$, 
and so $\varphi(T)$ has no multiple factors over the residue class field of $P$. 
Thus, by Kummer's Theorem \cite{stic2}, $P$ is unramified in $E$. 
By Abhyankar's Lemma \cite{stic2}, any extension of $P$ in $F_k$ is then 
unramified in $E_k$ for all $k\geq 1$. Thus, the  $s$ and $r(k)$ defined in  Lemma \ref{lemGenus} are as follows: 
\[s=r(k)=m-1 \textrm{ for all $k\geq1$}.\]
Now it follows from  Lemma \ref{lemGenus} and (c) that

\begin{equation}\label{gEk}
g(E_k)=\bigg(\frac{mq-m-1}{2}\bigg)d^k-\frac{m}{2}a_k+\frac{m+1}{2} 
\textrm{ with } \lim_{k\to \infty}\frac{a_k}{[E_k:E]}=0.
\end{equation}
Next, we want to estimate $a_k$. By (a), (d) and the fact that $e(Q|P)\geq 2$ and 
$\deg P=1$ for all $P\in R({\mathcal F})$, we obtain that
\begin{equation}\label{ak}
\begin{array}{lll}
a_k&=&\sum_{P\in R({\mathcal F})}\sum_{\substack{Q\in \mathbb{P}(F_k)\\ Q|P}} \deg Q=
\sum_{P\in R({\mathcal F})}\sum_{\substack{Q\in \mathbb{P}(F_k)\\ Q|P}} f(Q|P)\deg P\\
 &=&\sum_{P\in R({\mathcal F})\setminus \{P_0\}}\sum_{\substack{Q\in \mathbb{P}(F_k)\\ Q|P}} \deg Q+
 \sum_{\substack{Q\in \mathbb{P}(F_k)\\ Q|P_0}} f(Q|P_0)\deg P_0\\
 &=& \sum_{P\in R({\mathcal F})\setminus \{P_0\}}\sum_{\substack{Q\in \mathbb{P}(F_k)\\ Q|P}} f(Q|P)\deg P+1 \\
 &\leq& \sum_{P\in R({\mathcal F})\setminus \{P_0\}}\sum_{\substack{Q\in \mathbb{P}(F_k)\\ Q|P}} 2 f(Q|P)+1 \\
&\leq & \sum_{P\in R({\mathcal F})\setminus \{P_0\}}\sum_{\substack{Q\in \mathbb{P}(F_k)\\ Q|P}} e(Q|P) f(Q|P)+1 \\
&\leq & (\# R({\mathcal F})-1)[F_k:F_0]+1=(q-1)d^k+1,\\
& & \textrm{ by Fundamental Equality \cite{stic2}}. 
\end{array}
\end{equation}
Hence, by (\ref{ak}),
\begin{equation}\label{ak1}
a_k= \sum_{P\in R({\mathcal F})\setminus \{P_0\}}\sum_{\substack{Q\in \mathbb{P}(F_k)\\ Q|P}} f(Q|P)\deg P+1\leq \frac{(q-1)}{2}d^k+1 
\end{equation}
On the other hand as $f(Q|P)\geq 1$, by (\ref{ak}),
\begin{equation}\label{ak2}
a_k\geq (\#R({\mathcal F})-1)+1=q
\end{equation}
Therefore, by  (\ref{ak1}) and (\ref{ak2}), 
\begin{equation}\label{ak3}
q\leq a_k\leq \frac{(q-1)}{2}d^k+1.
\end{equation}
Now (ii) follows from (\ref{gEk}) and (\ref{ak3}). 
(iii) follows from (i) and (ii). 
\end{proof}

\begin{proposition}
The steps of the tower 
$\mathcal{E}/{\mathbb F}_{q}=(E_k/{\mathbb F}_{q})_{k\geq 0}$ 
defined in Proposition \ref{param1bis} have a class number such that
$$h(E_k/{\mathbb F}_{q})\geq h_{BRT1,k}$$ 
where
\begin{equation*}
\begin{split}
h_{BRT1,k}= 
\frac{(q-1)^2}{(g_k+1)(q+1)-1} 
\bigg(
\prod_{i\in N\cup \{1\}} 
 \left(\begin{array}{c}
 B_{i,k}+m_i\\
B_{i,k}
\end{array}\right)
 \\
 + 
q^{g_k-1} \prod_{i\in N \cup \{1\} } \left [ \left( \frac{q^{i}}{q^{i}-1}\right)^{B_{i,k}}
-  \right. \\ \left. B_{i,k}\left(
\begin{array}{c}
 B_{i,k}+m_i\\
 B_{i,k}
\end{array}
\right) \int_{0}^{\frac{1}{q^{i}}} \frac{(\frac{1}{q^{i}}-t)^{m_i}}{(1-t)^{B_{i,k}+m_i+1}}\rm dt. \right ]
\biggr)
\end{split}
\end{equation*}
 where $B_{i,k}= d^k$ for all $i\in N$,  
$g_k= \lceil (\frac{mq-m-2}{4})d^k+\frac{1}{2} \rceil$ and $\sum_{i\in N\cup\{1\}} im_i\leq g_k-2$ for all $k\geq 1$.
\end{proposition}

\begin{proof}
The bound $h_{BRT1,k}$ follows directly from Theorem \ref{mainmain} by taking 
$l_i=m_i$ for $i\in N\cup \{1\}$ and the lower bounds for the genus and for the number of places of each degree given by 
Proposition \ref{param1bis}.
\end{proof}

\medskip

{\bf Numerical estimations:} 

\medskip

Here are some effective numerical estimations of  the class numbers of some steps $E_k/{\mathbb F}_{q}$ of the tower 
$\mathcal{E}/{\mathbb F}_{q}$ given by Proposition \ref{param1bis} with $q=4$, $d=3$, and $N=\{2,3\}$.
These estimations are given by using Proposition \ref{param1bis}.
 
\medskip

\begin{center}
\begin{tabular}{|l|c|c|c|c|c|c|c|c|r|}
\hline
$q$  & step k & $g(E_{k})$ &  $B_{1,k}$  & $B_{2,k}$ & $B_{3,k}$ & $m_1$& $m_2$ & $m_3$ & $h_{BRT1,k}$   \\
\hline
4 &  2  &    30 &    1   &   9 &  9  &   7& 6 & 3 &  4.625820113117 $\times 10^{16}$\\
 \hline
4 &    3   &    89  &  1   &   27  &   27&  18 & 18 &  11&   2.23693900061$\times 10^{52}$ \\
 \hline     
 \end{tabular}
\end{center}

\medskip

\noindent {\bf c) Example 3}:
 
\begin{proposition} \label{fermat1} Let ${\mathcal F}/{\mathbb F}_9=(F_k/{\mathbb F}_9)_{k\geq 0}$ be the tower defined by the polynomial
\begin{equation}\label{lr.poly}
f(X,Y)=Y^2+(X+b)^2-1 \in {\mathbb F}_9[X,Y] \quad \textrm{ with } \quad b\in {\mathbb F}^*_l
\end{equation}
 and $F_0={\mathbb F}_9(x_0)$ be the rational function field. 
Let $E=F_0(z)$ with $z$ a root of the polynomial 
\[ \varphi(T)=(T^2+\mu^7)(T^9-T)-\frac{1}{x_0} \in F_0[T],\]
where  $\mu$ is a primitive element for ${\mathbb F}_9$. Then the sequence ${\mathcal E}/{\mathbb F}_9=(E_k/{\mathbb F}_9)_{k\geq 0}$, 
with $E_k=EF_k$, is a composite quadratic tower such that
\begin{itemize}
\item[(i)] $B_1(E_k)\geq 9\cdot 2^k$ and $B_2(E_k)\geq 2^k $ for all $k\geq 0$.
\item[(ii)] for all $k\geq 0$,
 \[g(E_k)=\begin{cases} 21\cdot 2^{k-1}-33\cdot 2^{(k-2)/2}+6 \;\textrm{ if $k\equiv 0\mod 2$}\\
                    21\cdot 2^{k-1}-11\cdot 2^{(k+1)/2}+6 \;\textrm{ if $k\equiv 1\mod 2$.}
\end{cases}\]
\item[(iii)] $ \beta_1({\mathcal E}/{\mathbb F}_9)=\frac{6}{7},\; \beta_2({\mathcal E}/{\mathbb F}_9)=\frac{2}{21}$ and $\beta_r({\mathcal E}/{\mathbb F}_9)=0$ for all  $r\neq 1,2$. 
\end{itemize}
\end{proposition}
\begin{proof} By \cite[Theorem 3.11]{gast2} and \cite[Proposition 3.3.12]{tutd}, 
the tower ${\mathcal F}/{\mathbb F}_4$ has the following properties: 
\begin{itemize}
\item[(a)] $\mathcal{R}({\mathcal F})=\{P\in \mathbb{P}(F_0): x_0(P)=\alpha \textrm{ for some } \alpha \in {\mathbb F}_3\}$.
\item[(b)] The pole (resp. the zero) of $x_0$, say $P_\infty$ (resp. $P_0$), 
splits completely (resp. is totally ramified) in ${\mathcal F}/{\mathbb F}_q$. 
\item[(c)] For all $k\geq 0$, 
\begin{eqnarray*}
g(F_k)=\begin{cases} 2^{k-1}-3\cdot 2^{(k-2)/2}+1 \qquad  \textrm{ if } k\equiv 0 \mod 2\\
                     2^{k-1}-2^{(k+1)/2}+1 \qquad \quad \; \textrm{ if } k\equiv 1 \mod 2,
        \end{cases}
\end{eqnarray*}
\item[(d)] ${\mathcal F}/{\mathbb F}_9$ is maximal with $\beta_1({\mathcal F}/{\mathbb F}_9)=2$.
\end{itemize}
By applying Proposition \ref{construction} with the place $Q=P_0$ and $S=\{P_\infty\}$, 
we obtain that the sequence ${\mathcal E}/{\mathbb F}_9$ is a quadratic  tower and the assertion (i) holds. \\
(ii) We will again apply Lemma \ref{lemGenus}. One can easily check that $(\varphi(T), \varphi'(T))=1$ 
at the places $P\in \mathbb{P}(F)$ with $x_0(P)=\alpha$ for some $\alpha \in {\mathbb F}_3^*$. Hence,  $\varphi(T)$  
has no multiple factor over the residue class field of these places. Thus, 
it follows from Kummer's Theorem \cite{stic2}  that these places are unramified in $E$. 
Then by Dedekind's Different Theorem \cite{stic2},  the $s$ in Lemma \ref{lemGenus} is 
\begin{equation}\label{delta.last}
s=10 \quad \textrm{ where $Q_0$ is the extension of $P_0$ in $E$.}
\end{equation}
By using Abhyankar's Lemma \cite{stic2}, we have that all 
$P_k\in \mathbb{P}(F_k)$ with $P_k\cap F\in \{P\in \mathbb{P}(F): x_0(P)\in {\mathbb F}_3^*\}$ 
are unramified in $E_k$. Moreover, since $P_0$ is totally ramified in both extensions 
$E/F_0$ and $F_k/F_0$, again by the same lemma, we have that any extension of $P_0$ in $F_k$ is 
totally ramified in $E_k$. Hence,
\begin{equation} \label{rn.last}
r(k)=10 \textrm{ for all $k\geq 1$}.
\end{equation}
Now by applying Lemma \ref{lemGenus}, (\ref{delta.last}), (\ref{rn.last}), and (c), 
we obtain the desired result.\\
(iii) follows from (i), (ii), (d) and \cite[Theorem 3.7(ii)]{hesttu}. 
\end{proof}

\begin{proposition}\label{hfermat1}
The steps of the tower 
$\mathcal{E}/{\mathbb F}_{9}=(E_k/{\mathbb F}_{9})_{k\geq 0}$ 
defined in Proposition \ref{fermat1} have a class number such that
$$h(E_k/{\mathbb F}_{9})\geq h_{BRT1,k}$$ 
where
\begin{equation*}
\begin{split}
h_{BRT1,k}= 
\frac{64}{10(g_k+1)-9.2^k} 
\bigg(
\prod_{i\in \{1,2\}} 
 \left(\begin{array}{c}
 B_{i,k}+m_i\\
B_{i,k}
\end{array}\right)
 \\
 + 
9^{g(E_k)-1} \prod_{i\in \{1,2\} } \left [ \left( \frac{9^{i}}{9^{i}-1}\right)^{B_{i,k}}
-  \right. \\ \left. B_{i,k}\left(
\begin{array}{c}
 B_{i,k}+m_i\\
 B_{i,k}
\end{array}
\right) \int_{0}^{\frac{1}{9^{i}}} \frac{(\frac{1}{9^{i}}-t)^{m_i}}{(1-t)^{B_{i,k}+m_i+1}}\rm dt. \right ]
\biggr)
\end{split}
\end{equation*}
 where $B_{1,k}= 9.2^k$, $B_{2,k}= 2^k$, and $m_1+2m_2\leq g(E_k)-2$ for all $k\geq 1$.

\end{proposition}

\begin{proof}
The bound $h_{BRT1,k}$ follows directly from Theorem \ref{mainmain} by taking 
$l_i=m_i$ for $i\in\{1,2\}$ and the lower bounds for the number of places of each degree given by 
Proposition \ref{fermat1}.
\end{proof}

\medskip

{\bf Numerical estimations:} 

\medskip

Here are some effective numerical estimations of  the class numbers of some steps $E_k/{\mathbb F}_{q}$ of the tower 
$\mathcal{E}/{\mathbb F}_{q}$ given by Proposition \ref{fermat1}. 
These estimations are given by using Proposition \ref{hfermat1}.

\medskip

\begin{center}
\begin{tabular}{|l|c|c|c|c|c|c|c|c|r|}
\hline
$q$  & step k & $g_k$ &  $B_{1,k}$  & $B_{2,k}$ & $m_1$ & $m_2$ & $h_{BRT1,k}$   \\
\hline
9 &  2  &    15 &     36 &  4  &   11&   1& 8.56396791978 $\times 10^{14}$   \\
 \hline
9 &    3   &   46   &   72  &   8  &  34 &   5& 7.47053482878  $\times 10^{45}$  \\
 \hline
 \end{tabular}      
\end{center}   

\medskip

\noindent {\bf d) Example 4}:

\begin{proposition}\label{fermat2} Let ${\mathcal F}/{\mathbb F}_4=(F_k/{\mathbb F}_4)_{k\geq 0}$ be the tower defined by the polynomial
\begin{equation}\label{lr.poly2}
f(X,Y)=Y^3+(X+1)^3-1 \in {\mathbb F}_4[X,Y]
\end{equation}
 and $F_0={\mathbb F}_4(x_0)$ be the rational function field. 
Let $E=F_0(z)$ with $z$ a root of the polynomial 
              \[ \varphi(T)=(T^4-T)(T^2+\mu T+\mu)-\frac{1}{x_0}\in F_0[T],\]
where $\mu$ is a primitive element for ${\mathbb F}_4$. 

Then the sequence ${\mathcal E}/{\mathbb F}_4=(E_k/{\mathbb F}_4)_{k\geq 0}$, with  $E_k=EF_k$, is a composite tower such that
\begin{itemize}
 \item[(i)] $B_1(E_k)\geq 4\cdot3^k$  and  $B_2(E_k)\geq 3^k$for all $k\geq 0$,
 \item[(ii)] $g(E_k)=\begin{cases} 
                         \frac{13}{2}\cdot3^{k}-4\cdot 3^{(k+2)/2}+\frac{5}{2} \quad \; \textrm{ if } k\equiv 0 \mod 2\\
                         \frac{13}{2}\cdot 3^{k} -2\cdot3^{(k+3)/2}+ \frac{5}{2}\quad  \;\textrm{ if } k\equiv 1 \mod 2. 
                         \end{cases}$
\item[(iii)] $\beta_1({\mathcal E}/{\mathbb F}_4)=\frac{8}{13}$, $\beta_2({\mathcal E}/{\mathbb F}_4)=\frac{2}{13}$ and $\beta_r({\mathcal E}/{\mathbb F}_4)=0$ for all $r\neq 1,2$.
\end{itemize}
\end{proposition} 

\begin{proof} 
By \cite[Theorem 3.11]{gast2} and \cite[Proposition 3.3.12]{tutd}, 
the tower ${\mathcal F}/{\mathbb F}_4$ has the following properties: 
\begin{itemize}
\item[(a)] $\mathcal{R}({\mathcal F})=\{P\in \mathbb{P}(F_0): x_0(P)=\alpha \textrm{ for some } \alpha \in {\mathbb F}_4\}$.
\item[(b)] The pole (resp. the zero) of $x_0$, say $P_\infty$ (resp. $P_0$), 
splits completely (resp. is totally ramified) in ${\mathcal F}/{\mathbb F}_q$. 
\item[(c)] For all $k\geq 1$, 
\begin{eqnarray*}
g(F_k)=\begin{cases} 3^k-2\cdot 3^{k/2}+1 \qquad  \textrm{ if } k\equiv 0 \mod 2\\
                     3^k-3^{(k+1)/2}+1 \qquad  \textrm{if } k\equiv 1 \mod 2.
        \end{cases}
\end{eqnarray*}
\item[(d)] The tower ${\mathcal F}/{\mathbb F}_4$ is maximal with $\beta_1({\mathcal F}/{\mathbb F}_4)=1$.
\end{itemize}
By applying Kummer's Theorem \cite{stic2} and  Proposition \ref{construction} with 
the set $S=\{P_\infty\}$ and $Q=P_0$, we obtain the following: 
\begin{itemize}
\item[(1)] $[E:F_0]=\deg \varphi(T)=6$ for all $k\geq 0$, 
\item[(2)] $P_0$ is totally ramified in $E$, and so ${\mathbb F}_4$ is algebraically closed in $E$,
\item[(3)] $P_\infty$ has four rational extensions and one extension of 
degree two in $E$, and they all split completely in $E_k$ for all $k\geq 1$,
\item[(4)] Since $(\varphi(T),\varphi'(T))=1$ at the places $P_{\mu}=(x_0=\mu)$ 
and $P_{\mu^2}=(x_0=\mu^2)$, these places are unramified in $E$.
\item[(5)] $P_1=(x_0=1)$ has exactly one extension which has degree $2$ and ramification index $3$ (by using Magma). 
\end{itemize}
By (3), $E_k/{\mathbb F}_4$ has an ${\mathbb F}_4$-rational place, and so ${\mathbb F}_4$ is algebraically closed in $E_k$ for all $k\geq 0$. 
By \cite[Lemma 3.3.9(i)]{tutd}, the place $(x_0=\mu^2)$ is totally ramified in $F_k$. 
Let $Q$ be an extension of $(x_0=\mu^2)$ in $E$. Then using (4) and  Abhyankar's Lemma \cite{stic2} 
gives that $Q$ has an extension, say $Q'$ in $E_k$ with the ramification index such that
\[[E_k:E]\geq e(Q'|Q)=[F_k:F_0]\geq [E_k:E].\]
Thus, $[E_k:E]=[F_k:F_0]$, i.e.,  $E_k/F_0$ and $F_k/F_0$ are linearly disjoint. 
Consequently, the sequence ${\mathcal E}/{\mathbb F}_4$ is a composite tower. \\
(i) follows from (3): 
\begin{equation}\label{split12}
  B_1(E_k)\geq 4\cdot [E_k:E]=4\cdot3^k \textrm{ and  $B_2(E_k)\geq [E_k:E]=3^k$ for all $k\geq 0$.}
\end{equation}
(ii)  We will use the Hurwitz Genus Formula \cite{stic} for the extension $E_k/F_k$ for any $k\geq 1$. 
It is clear that $g(E)=0$. By applying the  Hurwitz 
Genus Formula for the extension $E/F_0$, we get that the degree of the different of $E/F_0$ is 
 \begin{equation}\label{diff}
  \deg \operatorname{Diff}(E/F_0)=10.
\end{equation} 
  Let $Q\in \mathbb{P}(E)$. By using (2), (5), and Dedekind's Different Formula \cite{stic2},  we have that
\begin{itemize}
\item[(*)]  if $Q|P_1$, then $d(Q|P_1)=e(Q|P_1)-1=2$. 
\item[(**)] If $Q|P_0$, then $d(Q|P_0)>e(Q|P_0)-1=5$. Moreover, by (*), (5), and (\ref{diff}), 
we have that  $6\geq d(Q|P_0)$. Hence, $d(Q|P_0)=6$.
\end{itemize}  
Notice that  (\ref{diff}), (*), and  (**) then imply that only $P_0$ and $P_1$ are ramified in $E/F_0$. 
Now let $P_k$ be a place of $F_k$, $P:=P_k\cap F_0$ and $Q_k$ be an extension of $P_k$ in $E_k$. Set $Q:=Q_k\cap E$. Then clearly $Q|P$. 
 By Abhyankar's Lemma \cite{stic}, if $e(Q_k|P_k)>1$, then $e(Q|P)>1$.  Hence, it is enough to consider  just the places $P_0$ and $P_1$. First suppose that $P:=P_1$. By \cite[Lemma 3.3.9]{hesttu}, $P$ is totally  ramified in $F_k$. Thus, using (5) and Abhyankar's lemma, we get that $e(Q_k|P_k)=1$. 
 
 Now suppose that $P:=P_0$. By (b), (2) and Abhyankar's Lemma, we have that 
 \begin{equation}\label{indices}
 e(Q_k|P_k)=2 \quad \textrm{ and } \quad e(Q_k|Q)=3^{k-1}.
 \end{equation}
Thus, by Dedekind's Different Theorem \cite{stic}
\begin{equation}\label{dif.}
d(Q_k|Q)=e(Q_k|Q)-1=3^{k-1}-1 \quad \textrm{ and }\quad d(P_k|P)=e(P_k|P)-1=3^k-1.
\end{equation}
Now using (**), (\ref{indices}), (\ref{dif.}) 
and applying Transitivity of Different \cite{stic} in 
$F_0\subseteq E\subseteq E_k$ and $F_0\subseteq F_k\subseteq E_k$, we obtain that
$$e(Q_k|P_k)d(P_k|P)+d(Q_k|P_k)=2(3^k-1)+d(Q_k|P_k)$$
$$= e(Q_k|Q)d(Q|P)+d(Q_k|Q)=6\cdot 3^{k-1}+3^{k-1}-1.$$
Hence,
\begin{equation}\label{dif.2}
d(Q_k|P_k)=3^{k-1}+1.
\end{equation}
Now using (\ref{indices}), (\ref{dif.2}), and the Hurwitz Genus Formula for the extension $E_k/F_k$, we obtain that
\begin{eqnarray*}
2g(E_k)-2&=&6(2g(F_k)-2)+\deg \operatorname{Diff}(E_k/F_k))\\
         &=&6(2g(F_k)-2)+\sum_{Q_k|P_k} d(Q_k|P_k)\deg (Q_k)\\
         &=&6(2g(F_k)-2)+\sum_{Q_k|P_k} (3^{k-1}+1)f(Q_k|P_k)\\
         &=&6(2g(F_k)-2)+ 3\cdot (3^{k-1}+1) \quad \textrm{ by Fundamental Equality \cite{stic}}\\
         &=& 12 g(F_k)-12+3^k+3.
\end{eqnarray*}
Hence,
\begin{equation}
g(E_k)=6g(F_k)+\frac{3^{k}}{2}-\frac{7}{2}.
\end{equation}
 Now, combining this with (c) yields the desired result for $g(E_k)$ .\\
 (iii) follows from (i), (ii), (d) and \cite[Theorem 3.7(ii)]{hesttu}. 
\end{proof}

\begin{proposition}\label{hfermat2}
The steps of the tower 
$\mathcal{E}/{\mathbb F}_{4}=(E_k/{\mathbb F}_{4})_{k\geq 0}$ 
defined in Proposition \ref{fermat1} have a class number such that
$$h(E_k/{\mathbb F}_{4})\geq h_{BRT1,k}$$ 
where
\begin{equation*}
\begin{split}
h_{BRT1,k}= 
\frac{9}{5(g_k+1)-4.3^k} 
\bigg(
\prod_{i\in \{1,2\}} 
 \left(\begin{array}{c}
 B_{i,k}+m_i\\
B_{i,k}
\end{array}\right)
 \\
 + 
4^{g(E_k)-1} \prod_{i\in \{1,2\} } \left [ \left( \frac{4^{i}}{4^{i}-1}\right)^{B_{i,k}}
-  \right. \\ \left. B_{i,k}\left(
\begin{array}{c}
 B_{i,k}+m_i\\
 B_{i,k}
\end{array}
\right) \int_{0}^{\frac{1}{4^{i}}} \frac{(\frac{1}{4^{i}}-t)^{m_i}}{(1-t)^{B_{i,k}+m_i+1}}\rm dt. \right ]
\biggr)
\end{split}
\end{equation*}
 where $B_{1,k}= 4.3^k$, $B_{2,k}= 3^k$, and $m_1+2m_2\leq g(E_k)-2$ for all $k\geq 1$.

\end{proposition}

\begin{proof}
The bound $h_{BRT1,k}$ follows directly from Theorem \ref{mainmain} by taking 
$l_i=m_i$ for $i\in\{1,2\}$ and the lower bounds for the number of places of each degree given by 
Proposition \ref{fermat2}.
\end{proof}

\medskip

{\bf Numerical estimations:} 

\medskip

Here are some effective numerical estimations of  the class numbers of some steps $E_k/{\mathbb F}_{q}$ of the tower 
$\mathcal{E}/{\mathbb F}_{q}$ given by Proposition \ref{fermat2}. 
These estimations are given by using Proposition \ref{hfermat2}.

\medskip

\begin{center}
\begin{tabular}{|l|c|c|c|c|c|c|c|c|r|}
\hline
$q$  & step k & $g_k$ &  $B_{1,k}$  & $B_{2,k}$ & $m_1$ & $m_2$ & $h_{BRT1,k}$   \\
\hline
4 &  2  &    25 &     36 &  9  &   19&   2& 1.41572226696 $\times 10^{18}$   \\
 \hline
4 &    3   &   124   &   108  &   27  &  90 &   16& 3.50178913096  $\times 10^{86}$  \\
 \hline
 \end{tabular}      
\end{center}   
 
\medskip   
        
\subsection{infinitely many}

\begin{proposition}\label{infiniplace}
Let ${\mathbb F}_q$ be a finite field of characteristic $p$ and $m\in \mathbb{N}$ with $(m,p)=1$ and $m\geq 2$. 
Then there is a tower $\mathcal{F}/{\mathbb F}_q=(F_k/{\mathbb F}_q)_{k\geq 0}$ and a strictly increasing 
sequence $(n_i)_{i\geq 0}$ of positive integers with $n_0=1$ such that
\begin{itemize}
\item[(i)] $B_{n_i}(F_k)\geq p^{k-i}$ for all $0\leq i\leq k$,
\item[(ii)] $g(F_k)\geq \frac{m-1}{2} (p^k-1)$ for all $k\geq 0$.
\end{itemize}
\end{proposition}

\begin{proof} (i) The proof is similar to that of \cite[Lemma 3.17]{hesttu}.   Let $F_0={\mathbb F}_q(x_0)$ 
be the rational function field. Set 
\[ \textrm{ $Q_0=(x_0=\infty)$, $S_0=\{P_0\}$ where $P_0=(x_0=0)$, and $n_=1$.}\]
Choose an element $z_0\in F_0$ with the following properties:
\[z_0(P)=0 \textrm{ for } P\in S_0 \textrm{ and } v_{Q_0}(z_0)=-m.\]
Note that by the Weak Approximation Theorem \cite{stic2} such an element $z_0$ always exists. 
Let $F_1=F_0(x_1)$ where $x_1$ satisfies the equation 
\[x_1^p-x_1=z_0.\]
Then $Q_0$ is totally ramified in $F_1$, and hence ${\mathbb F}_q$ is algebraically closed in $F_1$. 
Denote by $Q_1$ the place of $F_1$ lying above $Q_0$. Then by Artin-Schreier Extension Theorem \cite{stic2}, 
\begin{equation*}
d(Q_1|Q_0)=(m+1)(p-1).
\end{equation*}
Set 
\[S_1=\{P\in \mathbb{P}(F_1): P\cap F_0\in S_0\}\cup \{P_1\},\]
where $P_1 \in \mathbb{P}(F_1)$ having $\deg P_1=n_1$ for some $n_1>1$. 
Next, choose an element $z_1 \in F_1$ such that 
\[z_1(P)=0 \textrm{ for all } P\in S_1 \textrm{ and } v_{Q_1}(z_1)=-m \]
Let $F_2=F_1(x_2)$ where $x_2$ satisfies the equation
\[x_2^p-x_2=z_1.\]
Then $Q_1$ is totally ramified in $F_2$, and so ${\mathbb F}_q$ is algebraically closed in $F_2$. 
Moreover, again by Artin-Schreier Extension Theorem \cite{stic2},
\begin{equation*}
d(Q_2|Q_1)=(m+1)(p-1),
\end{equation*}
where $Q_2\in \mathbb{P}(F_2)$ lies above $Q_1$. We continue on this process inductively for $k\geq 2$. Set
\[S_{k-1}=\{P\in \mathbb{P}(F_{k-1}): P\cap F_{k-2} \in S_{k-2} \} \cup \{P_{k-1}\}\]
where $P_{k-1} \in \mathbb{P}(F_{k-1})$ of $\deg P_{k-1}=n_{k-1}$ for some $n_{k-1}>n_{k-2}$. 
Choose an element $z_{k-1} \in F_{k-1}$ such that the following hold:
\[ z_{k-1}(P)=0 \textrm{ for all } P\in S_{k-1} \textrm{ and } v_{Q_{k-1}}(z_{k-1})=-m \]
where $Q_{k-1} \in \mathbb{P}(F_{k-1})$ lies above $Q_{k-2}$. 
Then $Q_{k-1}$ is totally ramified in $F_k$, and hence ${\mathbb F}_q$ is algebraically closed in $F_k$. Moreover, 
\begin{equation}\label{d3}
d(Q_k|Q_{k-1})=(m+1)(p-1) \textrm{ where $Q_k\in \mathbb{P}(F_k)$ lies above $Q_{k-1}$.} 
\end{equation}
Now assuming that (ii) holds, we have  that $g(F_k)\to \infty$ as $k\to \infty$. 
Now it follows from the construction of $F_k$, for $k\geq 0$, that the 
sequence ${\mathcal F}/{\mathbb F}_q=(F_k/{\mathbb F}_q)_{k\geq 0}$ is a tower with $[F_k:F_{k-1}]=p$ 
for all $k\geq 0$. Moreover, for all $i\geq 0$, by Kummer's Theorem \cite{stic2}, 
each place $P\in S_i$ splits completely in $F_k$, for any $k\geq i$. Thus, for all $i\geq 0$, we obtain that
\[B_{n_i}(F_k)\geq [F_k:F_i]=p^{k-i}.\]
(ii) We prove by induction. For $k=0$, it is trivial. The rest of the proof follows by using  
(\ref{d3}) and  the Hurwitz Genus Formula.
\end{proof}

\begin{proposition}
Let ${\mathbb F}_q$ be a finite field of characteristic $p$ and $m\in \mathbb{N}$ with $(m,p)=1$ and $m\geq 2$. 
Then there is a tower $\mathcal{F}/{\mathbb F}_q=(F_k/{\mathbb F}_q)_{k\geq 0}$ and a strictly increasing 
sequence $(n_i)_{i\geq 0}$ of positive integers with $n_0=1$ such that
the steps of this tower 
$\mathcal{F}/{\mathbb F}_{q}=(F_k/{\mathbb F}_{q})_{k\geq 0}$ 
have a class number such that
$$h(F_k/{\mathbb F}_{q})\geq h_{BRT1,k}$$ 
where
\begin{equation}\label{HBRT1}
\begin{split}
h_{BRT1,k}= 
\frac{(q-1)^2}{(g_k+1)(q+1)-p^k} 
\left(
\prod_{i=0}^{k} \left(\begin{array}{c}
 p^{k-i}+l_i\\
l_i
\end{array}\right) \right .
 \\
\left . +q^{g_k-1}
\prod_{i=0}^{k}\left [ \left( \frac{q^{n_i}}{q^{n_i}-1}\right)^{p^{k-i}}
-p^{k-i}\left(
\begin{array}{c}
 p^{k-i}+m_i\\
 p^{k-i}
\end{array}
\right) \int_{0}^{\frac{1}{q^{n_i}}} \frac{(\frac{1}{q^{n_i}}-t)^{m_i}}{(1-t)^{p^{k-i}+m_i+1}}\rm dt. \right ]
\right).
\end{split}
\end{equation}

where $g_k= \frac{m-1}{2} (p^k-1)$ and $\sum_{i=0}^{k}n_im_{i}\leq g_k-2$ for all $k\geq 0$.
\end{proposition}

\section{Annexe}
The following table gives the best current bounds according to
the known information on the number of places and the size of the genus.
Recall that the bounds $h_{LMD}$, $h_{BR}$, $h_{AHL}$  and $h_{BRT}$
are respectively defined by (\ref{HLMD}), (\ref{minh}), (\ref{HAHL}) and (\ref{HBRT}).

\medskip

\begin{center}
\begin{tabular}{||c|c|l||}
\hline \hline
Known information & case & reference \\
\hline \hline
& & \\
None specific information & & \phantom{i} $h_{LMD}$ \, \cite{lamd}\\
& & \\
\hline
One degree $r$ with known $B_r$ & 
\begin{tabular}{c}
\\
 $r=1$ \\
\\
\\
\\
$r > 1$ 
\end{tabular}
& 
\begin{tabular}{l}
\\
$h_{AHL}$ \, \cite{auhala1} (low $g$\\ 
and $B_1\geq (\sqrt{q}-1)g+1$)\\
$h_{BR}$ \, \cite{baro5} (else)\\
\\
$h_{BRT}$
\end{tabular}\\
& & \\
\hline
& & \\
Several degrees $r_i$ with known $B_{r_i}$ & & \phantom{i} $h_{BRT}$\\
& & \\
\hline \hline
\end{tabular}
\end{center}

\medskip

In case we only know the number of places of degree one,
as noted previously $h_{BR}=h_{BRT}$ (see Remark \ref{comparaisontheorique}). 
In this case, for very low values of $g$ and only if the number of rational points $B_1$ is 
such that $B_1\geq (\sqrt{q}-1)g+1$, bound $h_{AHL}$ is well suited else for almost values of $g$,
bound $h_{BR}$ is better.
In case we only know the number of places of degree $r>1$, 
the bound $h_{BRT}$ is better than $h_{BR}$.
In case of several distinct degrees, bound $h_{BRT}$ gives the best results.

\section{Acknowledgments}
Seher Tutdere is partially supported by T\"UB\.ITAK under Grant No. $TBAG-109T672$.

\end{document}